\documentclass[12pt,letterpaper]{article}
\pdfoutput=1

\usepackage{graphics,epsfig,graphicx,color}
\graphicspath{{/EPSF/}{../Figures/}{Figures/}}
\usepackage{subfigure}
\usepackage{amssymb,latexsym}

\usepackage{setspace}

\usepackage{amsmath}

\usepackage{mathrsfs,amsfonts}

\usepackage{amsthm,amsxtra}

\usepackage{enumitem}

\newif\ifPDF
\ifx\pdfoutput\undefined
	\PDFfalse
\else
	\ifnum\pdfoutput > 0
		\PDFtrue
	\else
		\PDFfalse
	\fi
\fi

\ifPDF
	\usepackage{pdftricks}
	\begin{psinputs}
		\usepackage{pstricks}
		\usepackage{pstcol}
		\usepackage{pst-plot}
		\usepackage{pst-tree}
		\usepackage{pst-eps}
		\usepackage{multido}
		\usepackage{pst-node}
		\usepackage{pst-eps}
	\end{psinputs}
\else
	\usepackage{pstricks}
\fi

\ifPDF
	\usepackage[debug,pdftex,colorlinks=true, 
	linkcolor=blue, bookmarksopen=false,
	plainpages=false,pdfpagelabels]{hyperref}
\else
	\usepackage[dvips]{hyperref}
\fi

\usepackage{fancyhdr}

\newtheorem{theorem}{Theorem}[section]
\newtheorem{lemma}[theorem]{Lemma}

\newcommand{\be}{\begin{equation}}
\newcommand{\ee}{\end{equation}}

\newcommand{\bzero}{{\mathbf 0}}

\newcommand{\bbR}{\mathbb R}

\newcommand{\E}{\epsilon} 
\newcommand{\T}{\tau}

\newcommand{\bT}{{\bar\tau}}

\newcommand{\pushright}[1]{\ifmeasuring@#1\else\omit\hfill$\displaystyle#1$\fi\ignorespaces}
\newcommand{\pushleft}[1]{\ifmeasuring@#1\else\omit$\displaystyle#1$\hfill\fi\ignorespaces}

\newenvironment{keywords}
{\noindent{\bf Key words.}\small}{\par\vspace{1ex}}

\title{{Bipodal structure in oversaturated random graphs}}

\author{
Richard Kenyon\thanks{Department of Mathematics, Brown University, Providence, RI 02912; rkenyon at math.brown.edu}
\and Charles Radin\thanks{Department of Mathematics, University of Texas, Austin, TX 78712; radin@math.utexas.edu}  
\and Kui Ren \thanks{Department of Mathematics and ICES, University of Texas, Austin, TX 78712; ren@math.utexas.edu} 
\and Lorenzo Sadun\thanks{Department of Mathematics, University of Texas, Austin, TX 78712; sadun@math.utexas.edu} 
}

\begin{document}
\maketitle

\begin{abstract} {We study the asymptotics of large simple graphs
    constrained by the limiting density of edges and the limiting
    subgraph density of an arbitrary fixed graph $H$.  We prove that,
    for all but finitely many values of the edge density, if the
    density of $H$ is constrained to be slightly higher than that for
    the corresponding Erd\H{o}s-R\'enyi graph, the typical large graph
    is bipodal with parameters varying analytically with the
    densities. Asymptotically, the parameters depend only on the
    degree sequence of $H$. }
\end{abstract}


\section{Introduction}
\label{SEC:Intro}

We study the asymptotics of large, simple, labeled graphs constrained
to have subgraph densities $\E$ of edges, and $\T$ of some fixed
subgraph $H$ with $\ell \ge 2$ edges.  To study the asymptotics we
use the graphon formalism of Lov\'asz et al \cite{LS1, LS2, BCLSV,
  BCL, LS3} and the large deviations theorem of Chatterjee and
Varadhan \cite{CV}, from which one can reduce the analysis to the
study of the graphons which maximize the entropy subject to the
density constraints \cite{RS1, RS2, RRS,
  KRRS}.  
See definitions in Section \ref{SEC:Notation}.

The \emph{phase space} is the subset of
$[0,1]^2$ consisting of accumulation points of all pairs of densities
$\bT=(\E,\T)$ achievable by finite graphs. (See Figure
\ref{FIG:phase_space} for the case where $H$ is a triangle.)
Within the phase space is the `Erd\H{o}s-R\'enyi curve' (ER curve)
$\{(\E,\T)~|~\T=\E^\ell\}$, attained when edges are chosen
independently.  In this paper we study the typical behavior of large
graphs for $\T$ just above the ER curve. We will show that the
qualitative behavior of such graphs is the same for all choices of $H$
and for all but finitely many choices of $\E$ depending on $H$.

\begin{figure}[ht]
\centering
\includegraphics[angle=0,width=0.4\textwidth]{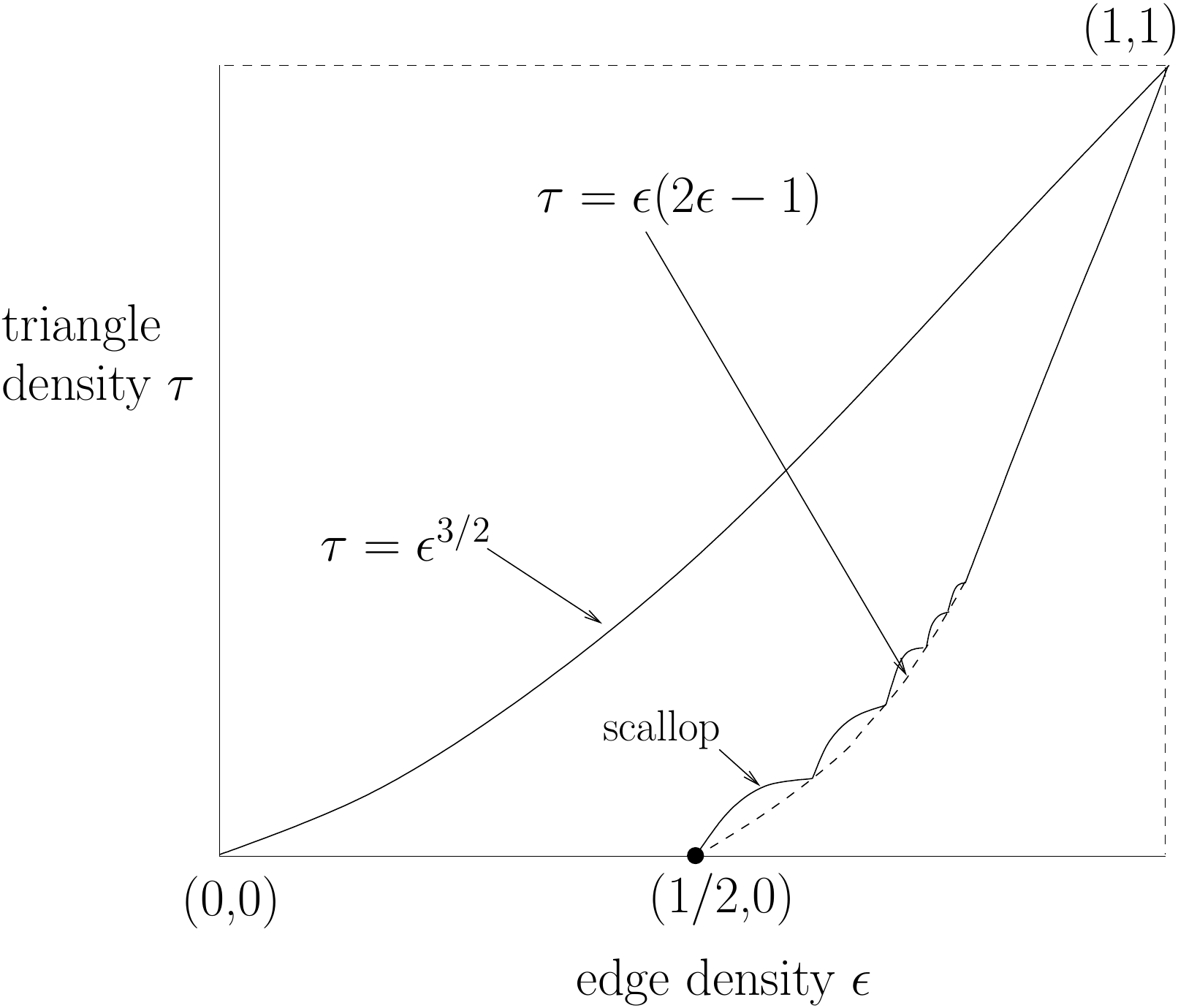}\hskip
1cm 
\includegraphics[angle=0,width=0.39\textwidth]{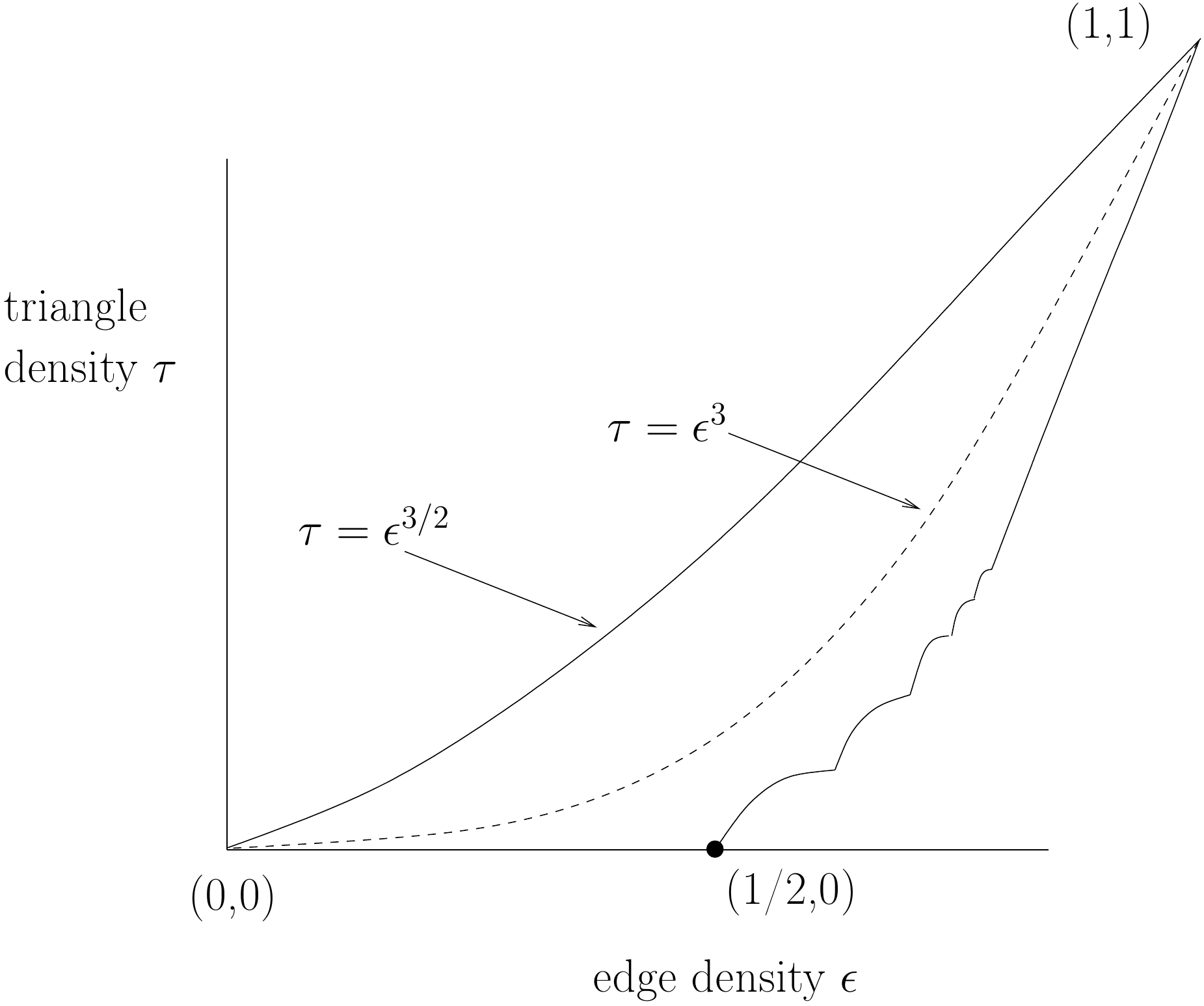} 
\caption{Boundary of the phase space for the edge/triangle model in
  solid lines. On
the right,
the Erd\H{o}s-R\'enyi curve is shown with dashes.}
\label{FIG:phase_space}
\end{figure}

To be precise, we show that for fixed $H$, for $\E$ outside a finite
set, and for $\T$ close
enough to $\E^\ell$, there is a {\it unique} entropy-maximizing
graphon (up to measure-preserving transformations of the unit
interval); furthermore it is bipodal and depends analytically on $(\E,\T)$,
implying that the entropy is an analytic function of $(\E,\T)$. In particular we prove 
the existence of one or more well-defined
thermodynamic phases just above the ER curve. {This is the first proof, as far as we know,
of the existence of a phase in any constrained-density graphon model,
where by \emph{phase} we mean a (maximal) open set in the phase space
where the 
entropy varies analytically with the constraint parameters. Conjecturally, phases form an open dense subset of the phase space.}

A \emph{bipodal graphon} is a function $g: [0,1]^2 \to [0,1]$ of the
form:
\be
 g(x,y) = \begin{cases} p_{11} & x,y < c, \cr 
  p_{12} & x<c<y, \cr p_{12} & y<c<x, \cr p_{22} & x,y >
  c. \end{cases} 
\ee
 Here $c,p_{11}, p_{12}$ and $p_{22}$ are
constants taking values between 0 and 1.  We prove that as $\T\searrow \E^\ell$, 
the parameters $c \to 0$, $p_{22} \to \E$, and
$p_{11}$ and $p_{12}$ approach the solutions of a problem in single-variable calculus. The inputs to that calculus problem depend only on
the degrees of the vertices of $H$.

We say that a finite graph $H$ is {\em $k$-starlike} if all the vertices
of $H$ have degree $k$ or 1, where $k >1$ is a fixed
integer. $k$-starlike graphs include $k$-stars (where one vertex has
degree $k$ and $k$ vertices have degree 1), and the complete graph on
$k+1$ vertices.  For fixed $k$, all $k$-starlike graphs behave
essentially the same for our asymptotics. We prove our results first for $k$-stars,
and then apply perturbation theory to show that the differences
between different $k$-starlike graphs are irrelevant, {and then prove the general case}.
 
To state our results more precisely, we need some notation. Let 
\begin{equation}
S_0(w) =  -\frac{1}{2}[w\log w +(1-w)\log(1-w)],
\end{equation}
and define the \emph{graphon entropy} {(or \emph{entropy} for short)}
of a graphon $g$ to be 
\begin{equation}
s(g) = \int_0^1\!\int_0^1  S_0(g(x,y)) dx\, dy.\end{equation}
Let  
\begin{equation} \label{firstrat}
\psi_k(\E,\tilde \E) =  \frac{2[S_0(\tilde \E)-S_0(\E) -S_0'(\E)(\tilde \E - \E)]}
{\tilde \E^k - \E^k -k\E^{k-1} (\tilde \E -\E )}.
\end{equation}
This function has a removable singularity at $\tilde \E=\E$, which we
fill by defining
\be
\psi_k(\E,\E) = \frac{2 S_0''(\E)}{k(k-1) \E^{k-2}}.\ee
For fixed $\E$, let $\zeta_k(\E)$ be the value of $\tilde \E$ that
maximizes $\psi_k(\E,\tilde \E)$. (We will prove that this maximizer
is unique and depends continuously on $\E$.)

\begin{theorem} \label{Main-Thm-Simple} Let $H$ be a $k$-starlike graph with
  $\ell\ge 2$
  edges.  Let $\E \in
  (0,1)$ be any point other than $(k-1)/k$. Then there is a number
  $\T_0> \E^\ell$ (depending on $\E$) such that for all $\T \in
  (\E^\ell, \T_0)$, the entropy-maximizing graphon at $(\E,\T)$ is
  unique (up to measure-preserving transformations of $[0,1]$) and
  bipodal. The parameters $(c, p_{11}, p_{12}, p_{22})$ are analytic
  functions of $\E$ and $\T$ on the region $\E \ne (k-1)/k$, $\T \in
  (\E^\ell, \T_0(\E))$. Furthermore, as $\T\searrow\E^\ell$ we have that
$p_{22} \to \E$, $p_{12} \to \zeta_k(\E)$, $p_{11}$ satisfies
  $S_0'(p_{11}) = 2S_0'(p_{12}) - S_0'(p_{22})$, and $c=O(\T-\E^\ell)$.
\end{theorem}

  Theorem \ref{Main-Thm-Simple} proves that there
  is part of a phase just above the ER curve for $\E < (k-1)/k$ and also for
  $\E > (k-1)/k$; numerical evidence suggests these are in fact parts
  of a single phase; the only `singular' behavior is the manner in which
  the graphon approaches the constant graphon associated with the ER curve. We will see in Theorem
  \ref{Main-Thm-Complex} that this behavior is only slightly
  more complicated for general $H$ than it is for $k$-starlike $H$.

When $H$ has vertices with different degrees $>1$, the problem
resembles that of a formal positive linear combination of
$k$-stars. As in the $k$-starlike case, we first solve the problem for
the linear combination of $k$-stars and then use perturbation theory
to extend the results to arbitrary $H$.

\begin{theorem} \label{Main-Thm-Complex} Let $H$ be an arbitrary graph
  with $\ell$ edges with at least one vertex of degree $2$ or
  greater. Then there exists a finite set $B_H \subset (0,1)$ such that if $\E \ne
  B_H$, then there is a number $\T_0> \E^\ell$ (depending on $\E$)
  such that for all $\T \in (\E^\ell, \T_0)$, the entropy-maximizing
  graphon at $(\E,\T)$ is unique (up to measure-preserving
  transformations of $[0,1]$) and bipodal. The parameters $(c, p_{11},
  p_{12}, p_{22})$ are analytic functions of $\E$ and $\T$ on the
  region $\E \not \in B_H$, $\T \in (\E^\ell, \T_0(\E))$. Furthermore,
  as $\T\searrow \E^\ell$ we have that $p_{22} \to \E$, $p_{12}$
  approaches the maximizer of an explicit function whose data depends
  on $\E$, $p_{11}$ satisfies $S_0'(p_{11}) = 2S_0'(p_{12}) -
  S_0'(p_{22})$, and $c=O(\T-\E^\ell)$.
\end{theorem}

The key differences between the Theorems \ref{Main-Thm-Simple} and
\ref{Main-Thm-Complex} are:
\begin{itemize}
\item For $k$-starlike graphs, the set $B_H$ of bad values of $\E$
  consists of a single point, and this point is explicitly known: $\E =
  (k-1)/k$.
\item For $k$-starlike graphs, the behavior of $\zeta_k$ is {explicit}. It is
  a continuous and strictly decreasing function of $\E$, and gives an
  involution of $(0,1)$. (That is, $\zeta_k(\zeta_k(\E))=\E$.) For
  $k=2$ it is given by $\zeta_2(\E)=1-\E$.  In the general
  case, the limiting value of $p_{12}$, and its dependence on $\E$,
  appear to be much more complicated. We do not know whether this
  limiting value is always continuous across the bad set $B_H$.
\end{itemize}

The organization of this paper is as follows.  In Section~\ref{SEC:Notation} we review the formalism of graphons and establish basic notation. In Section~\ref{SEC:k-star-prelim} we establish a number of technical results for $k$-star models. Using these results, in Section~\ref{SEC:k-stars} we prove Theorem~\ref{Main-Thm-Simple} for the case that $H$ is a $k$-star. In Section~\ref{SEC:k-simple} we show that just above the ER curve a model with an arbitrary $k$-starlike $H$ can be approximated by a $k$-star model. By bounding the error terms, we prove Theorem~\ref{Main-Thm-Simple} in full generality.  In Section~\ref{SEC:linear-com} we consider formal positive linear combinations of $k$-stars, and prove a theorem much like Theorem~\ref{Main-Thm-Complex} for those models. Finally, in Section~\ref{SEC:complex} we show that the model for an arbitrary $H$ can be approximated by a formal linear combination of $k$-stars, thus completing the proof of Theorem~\ref{Main-Thm-Complex}.

\section{Notation and background}
\label{SEC:Notation}

We consider a simple graph $G$ (undirected,
with no multiple edges or loops) with a vertex set $V(G)$ of labeled
vertices. For a subgraph $H$ of $G$, let $T_H(G)$ be the number of
maps from $V(H)$ into $V(G)$ which preserve edges. The \emph{density}
$\T_H(G)$ of $H$ in $G$ is then defined to be
\begin{equation}
\T_H(G):= \frac{|T_H(G)|}{n^{|V(H)|}},
\end{equation}
where $n = |V(G)|$. An important special case is where $H$ is 
a `$k$-star', a graph with $k$ edges, all with a
common vertex, for which we use the notation $\T_k(G)$. 
In particular $\T_1(G)$, which we also denote $\E(G)$, is the edge density of $G$.

For $\alpha > 0$ and $\bT=(\E,\T_H)$ define
$\displaystyle Z^{n,\alpha}_{\bT}$ to be the number of graphs $G$ on $n$
vertices with densities satisfying
\begin{equation}
\E(G) \in (\E-\alpha,\E+\alpha), \ \T_{H}(G) \in (\T_H-\alpha,\T_H+\alpha).
\end{equation}

Define the \emph{(constrained) entropy} $s_{\bT}$ to be the
exponential rate of growth of $Z^{n,\alpha}_{\bT}$ as a function of $n$:
\begin{equation}
s_{\bT}=\lim_{\alpha\searrow 0}\lim_{n\to \infty}\frac{\ln(Z^{n,\alpha}_{\bT})}{ n^2}.
\end{equation}
The double limit defining the entropy $s_{\bT}$ is known to
exist \cite{RS1}. To analyze it we make use of a variational
characterization of $s_{\bT}$, and for this we need further notation to
analyze limits of graphs as $n\to \infty$. (This work was recently
developed in \cite{LS1,LS2,BCLSV,BCL,LS3}; see also the recent book
\cite{Lov}.)  The (symmetric) adjacency matrices of graphs on $n$
vertices are replaced, in this formalism, by symmetric, measurable
functions $g:[0,1]^2\to[0,1]$; the former are recovered by using a
partition of $[0,1]$ into $n$ consecutive subintervals. The functions
$g$ are called graphons.

For a graphon $g$ define the \emph{degree function} $d(x)$ to be
$d(x)=\int^1_0 g(x,y)dy$.  The $k$-star density of $g$, $\T_k(g)$,
then takes the simple form
\begin{equation}
 \T_k(g) =  \int_0^1 d(x)^k\,dx.
\end{equation}
For any fixed graph $H$, the $H$-density $\T_H$ of $g$ can be similarly
expressed as an integral of a product of factors $g(x_i,x_j)$. 

The following is Theorem 4.1 in \cite{RS2}:
\begin{theorem}[The Variational Principle]
For any feasible set $\bT$ of values of the densities $\bT(g) := (\E, \T_H)$ 
we have 
$s_{\bT} = \max [s(g)]$, where the entropy is 
maximized over all graphons $g$ with $\bT(g)=\bT$.
\end{theorem}
\noindent (Instead of using $s(g)$, some authors use the \emph{rate function}
$I(g):= -s(g)$, and then minimize $I$.)
The existence of a maximizing 
graphon $g=g_{\bT}$ for any constraint $\bT(g)=\bT$ was proven in \cite{RS1},
again adapting a proof in \cite{CV}.
If the densities are that of edges and $k$-star subgraphs we refer
to this maximization problem as a \emph{star model}, though we
emphasize that the result applies much more generally \cite{RS1, RS2}.

We consider two graphs {\it equivalent} if they are obtained from one
another by relabeling the vertices. For graphons, the analogous
operation is applying a measure-preserving map $\psi$ of $[0,1]$ into
itself, replacing $g(x,y)$ with $g(\psi(x),\psi(y))$, see \cite{Lov}.
The equivalence classes of graphons under relabeling are called
\emph{reduced graphons}, and 
graphons are equivalent
if and only if they have the same subgraph densities for all possible
finite subgraphs \cite{Lov}. In the remaining sections of the paper,
whenever we claim that a graphon has a property (e.g. monotonicity in 
$x$ and $y$, or uniqueness as an entropy maximizer), the caveat ``up to 
relabeling'' is implied. 

The graphons which maximize the constrained entropy can tell us what
`most' or `typical' large constrained graphs are like: if $g_{\bT}$ is
the only reduced graphon maximizing $S(g)$ with $\bT(g)=\bT$, then as
the number $n$ of vertices diverges and $\alpha_n\to 0$, exponentially
most graphs with densities $\bT_i(G)\in (\T_i-\alpha_n,\T_i+\alpha_n)$
will have reduced graphon close to $g_{\bT}$ \cite{RS1}.  This is based
on large deviations from \cite{CV}. We emphasize that this
interpretation requires that the maximizer be unique; this has been
difficult to prove in most cases of interest and is an important focus
of this work.

A graphon $g$ is called $M$-podal if there is decomposition
of $[0,1]$ into $M$ intervals (`vertex clusters') $C_j,\ j=1,2,\ldots,M$, and
$M(M+1)/2$ constants $p_{ij}$ such that
$g(x,y)=p_{ij}$ if $(x,y)\in C_i\times C_j$ (and $p_{ji}=p_{ij}$). We 
denote the length of $C_j$ by $c_j$.

\section{Technical properties of star models}
\label{SEC:k-star-prelim}

For each star model, all entropy-maximizing graphons are multipodal
with a fixed upper bound on the number of clusters, also called the {\em podality}~\cite{KRRS}.  For any fixed
podality $M$, an $M$-podal graphon is described by $N=M(M+3)/2$
parameters, namely the values $p_{ij}$ ($1\le i\le j\le M$) and the
widths $c_i$ ($1\le i\le M$) of the clusters.
When it does not cause confusion, we
will use $g$ to denote the vector
\be(c_1,\cdots,c_{M},p_{11},\cdots,p_{1M},p_{22},\cdots,p_{2M},\cdots,\cdots,p_{M-1M-1},p_{M-1M},p_{MM}),\ee
which contains all these parameters. The problem of optimizing the
graphon then reduces to a finite-dimensional calculus problem. To be
precise, let us recall that for an $M$-podal graphon, we have
\begin{equation}
	 \E(g) = \sum_{1\le i,j\le M} c_ic_j p_{ij},\ \ 
\T_k(g) = \sum_{1\le i \le M} c_i d_i^k,\ \ 
s(g) = \sum_{1\le i,j\le M} c_i c_j S_0(p_{ij}),
\end{equation} 
where $d_i = \sum_{1\le j\le M} c_j p_{ij}$ is the value of the 
degree function on the $i$th cluster. 
The problem of searching for entropy-maximizing graphons with fixed edge density $\E$ and $k$-star density $\T_k$ can now be formulated as
\begin{equation}\label{EQ:Max}
	\max_{g\in [0, 1]^N} s(g), \quad \mbox{subject to:}\quad \E(g)-\E=0, \quad \T_k(g)-\T=0, \quad  C(g) =1. 
\end{equation}
where $C(g) = \sum_{1\le j\le M} c_j$. 

The following result says that the maximization problem~\eqref{EQ:Max}
can be solved using the method of Lagrange multipliers. The existence of finite
Lagrange multipliers was previously established in \cite{KRRS}, treating
the space of graphons as a linear space of functions $[0,1]^2 \to
[0,1]$, intuitively considering perturbations of graphons localized
about points in $[0,1]^2$.
For star models we may restrict to $M$-podal graphons, as noted above,
and thus consider perturbations in the relevant parameters
$p_{ij}$ and $c_j$.

\begin{lemma}
Let $g$ be a local maximizer in~\eqref{EQ:Max}. Then for constraints
$\E,\T$ off the ER curve, there exist unique $\alpha,\beta,\gamma\in\bbR$ such that
\begin{equation}\label{EQ:EL}
\nabla s(g)-\alpha \nabla \E(g)-\beta \nabla \T_k(g) - 
\gamma \nabla C(g)=\bzero.
\end{equation}
\end{lemma}

We do not include the proof, which follows easily from that of Lemma 3.5 in~\cite{KRRS}. We also note that one can remove the variable $c_M$ and the constraint $C(g) =1$, eliminating the multiplier $\gamma$.

For convenience later, we now write down the exact 
form of the Euler-Lagrange equation~\eqref{EQ:EL}. We first verify that
\begin{eqnarray}
\label{EQ:Grad e}\frac{\partial \E}{\partial p_{ij}} = A_{ij}, & & 
\frac {\partial \E}{\partial c_i} = 2 \sum_{j=1}^{M} c_j p_{ij}= 2 d_i,\\
\label{EQ:Grad t}\frac{\partial \T_k}{\partial p_{ij}} = 
\frac{k}{2}(d_i^{k-1} + d_j^{k-1}) A_{ij}, & & 
\frac{\partial \T_k}{\partial c_i}  = d_i^k + k\sum_{j=1}^M c_j d_j^{k-1} p_{ij},\\
\label{EQ:Grad C}\frac{\partial C}{\partial p_{ij}} = 0, & & \frac{\partial C}{\partial c_i} = 1,\\
\label{EQ:Grad s}\frac{\partial s}{\partial p_{ij}} = S_0'(p_{ij}) A_{ij}, & & \frac{\partial s}{\partial c_i} = 2 \sum_{j=1}^M c_j S_0(p_{ij}),
\end{eqnarray}
where $A_{ij}= 2 c_i c_j$ if $i\neq j$ and $A_{ij}= c_i^2$ if
$i=j$. We can then write down ~\eqref{EQ:EL} explicitly as
\begin{eqnarray}
\label{cmeq0}
S_0'(p_{ij}) &=& \alpha + \beta \frac{k}{2}(d_i^{k-1} + d_j^{k-1}),
\quad 1\le i\le j\le M\\\label{cmeq}
2\sum_{j=1} c_j S_0(p_{ij}) &=& 2\alpha d_i + \beta \big(d_i^k + k\sum_{j=1}^M c_j d_j^{k-1} p_{ij}\big)+\gamma,\quad 1\le i\le M
\end{eqnarray}
These Euler-Lagrange equations, together with the constraints,
\begin{equation}\label{EQ:Constraints}
\E(g) -\E=0, \qquad \T_k(g)-\T=0,\qquad C(g)-1=0,
\end{equation}
are the optimality
conditions for the maximization
problem~\eqref{EQ:Max}.  In principle, we can solve this system to
find the maximizer $g$.
 
Next we consider the significance of the Lagrange multipliers $\alpha$
and $\beta$. Suppose that $g_0$ is the unique entropy maximizer for
$\E=\E_0$ and $\T=\T_0$. Then any sequence of graphons that maximize
entropy for $(\E,\T)$ approaching $(\E_0,\T_0)$ must approach $g_0$: this follows from upper semicontinuity of the entropy and the fact that we can
perturb $g_0$ to any nearby $(\E,\T)$ by changing some $p_{ij}$.
But if $g= g_0 + \delta g$, then
\begin{eqnarray} s(g) & = & s(g_0) + ds_{g_0} (\delta g) + O(\delta g^2) \cr  
&=& s(g_0) + \alpha d\E_{g_0}(\delta g) + \beta d\T_{g_0}(\delta g)
+ O(\delta g^2) \cr 
& = & s(g_0) + \alpha (\E-\E_0) + \beta(\T-\T_0) + O(\delta g^2).
\end{eqnarray}
That is, $\partial s(\E,\T)/\partial \E = \alpha$ and $\partial
s(\E,\T)/\partial \T = \beta$.

If $g_0$ is not a unique entropy maximizer, then
we only have 1-sided {(directional)} derivatives: 

\begin{lemma}
  The function $s(\E,\T)$ admits directional derivatives in all
  directions at all points $(e,t)$ in the interior of the
  profile. 
\end{lemma}

\begin{proof}
The change in entropy in a given direction is obtained
by maximizing $ds=\alpha d\E + \beta d\T$ over all entropy maximizers at $(\E_0,\T_0)$. That is, when fixied $\E$ and increasing $\T$, we get the
largest $\beta$ of all the graphons that maximize entropy at $(\E_0,\T_0)$, 
and when decreasing $\T$ we get the smallest $\beta$. Likewise, when increasing
or decreasing $\E$ we get the largest or smallest values of $\alpha$,
and when doing a directional derivative in the direction $(v_1,v_2)$, we get
the largest value of $v_1 \alpha + v_2 \beta$. 
\end{proof}

Existence of directional derivatives implies the fundamental theorem of calculus, so for fixed $\E$ we can write
\begin{equation}\label{ftc}
 s(\E,\T) = s(\E, \E^k) + \int_{\E^k}^\T \beta(g_{max}(\E,\T)) d\T,
\end{equation}
where $g_{max}(\E,\T)$ is the {entropy-maximizing graphon at $(\E,\T)$ that maximizes its right derivative (with respect to $\T$)}. 

\medskip

Before proving Theorem \ref{Main-Thm-Simple} for $k$-stars, we {record}
some properties of the function $\psi_k(\E, \tilde \E)$  of \eqref{firstrat} and its critical
points. 

\begin{theorem}\label{thm-zeta}
For fixed $k$ and $\E$, there is a unique solution to $\partial \psi_k'(\E,\tilde \E)/\partial \tilde \E=0$, which we 
denote $\tilde \E=\zeta_k(\E)$. The function $\zeta_k$ is a strictly decreasing, 
with nowhere-vanishing derivative and 
with fixed point at $\E=(k-1)/k$.  Furthermore, $\zeta_k$ is an involution:
$\tilde \E = \zeta_k(\E)$ if and only if $\E = \zeta_k(\tilde \E)$. 
\end{theorem}

Even though the proof is elementary we will need some parts of it later, so we give it here.

\begin{proof}
Fix $k \ge 2$ and let 
\begin{eqnarray} N(\E,\tilde \E) & = & 2[S_0(\tilde \E) - S_0(\E) - 
S_0'(\E)(\tilde \E - \E)] \cr 
D(\E,\tilde \E) & = & \tilde \E^k - \E^k - k \E^{k-1}(\tilde \E - \E)
\end{eqnarray}
be the numerator and denominator of the function $\psi_k(\E, \tilde \E) = N/D$.
Note that these definitions make sense for all real values of $k$, not just
for integers. When taking derivatives of $N$, $D$ and $\psi$, we will denote
a derivative with respect to the first variable by a dot, and a derivative with respect to the
second variable by ${}'$. That is, 
$D'(\E,\tilde \E) =\partial D/\partial \tilde \E$ and 
$\dot D(\E, \tilde \E) = \partial D/\partial \E$. 
As noted earlier, this definition of $\psi_k$
has a removable singularity at $\tilde \E = \E$, which we fill in by defining
\be
 \psi_k(\E,\E) = N''(\E,\E)/D''(\E,\E) = 2S_0''(\E)/[k(k-1)\E^{k-2}].
\ee
The denominator
$D$ vanishes only at $\tilde \E = \E$. 

Some useful explicit derivatives are:
\begin{eqnarray}
& N' =  2[S_0'(\tilde \E) - S_0'(\E)], 
& N''  = 2 S_0''(\tilde \E) = \frac{-1}{\tilde \E (1-\tilde \E)}, \cr 
& \dot N  =  -2S_0''(\E)(\tilde \E - \E), \qquad 
& \dot N'  =  -2 S_0''(\E), \cr 
& D'  =  k[\tilde \E^{k-1} - \E^{k-1}], \qquad 
& D''  =  k(k-1) \tilde \E^{k-2}, \cr 
& \dot D  =  -k(k-1) \E^{k-2}(\tilde \E - \E), \qquad 
& \dot D'  =  -k(k-1) \E^{k-2}.
\end{eqnarray}
Note that $D$ and $N$ both
vanish when $\tilde \E = \E$, so we can write
\be
N(\E, \tilde \E)  = \int_\E^{\tilde \E} N'(\E,x) dx = \int_{\tilde \E}^\E
\dot N(x, \tilde \E) dx,
\ee
and similarly for $D(\E, \tilde \E)$.

We proceed in steps:
\begin{itemize}[leftmargin=\dimexpr 26pt+6mm]

\item[Step 1.] Analyzing $\psi$ near $\tilde \E = \E$ to see that 
$\psi_k'(\E,\E) = 0$ only when $\E = (k-1)/k$. 

\item[Step 2.] Showing that we can never have $\psi_k'=\psi_k''=0$. 

\item[Step 3.] Showing that the equation $\psi_k'(\E, \tilde \E)$ is 
symmetric in $\E$ and $\tilde \E$, implying that $\zeta_k$ is an involution. 

\item[Step 4.] Showing that $\psi_k$ has a unique critical point. 

\item[Step 5.] Showing that $d \zeta_k/d\E$ is never zero. 
\end{itemize}

The following calculus fact will be used repeatedly. When $D \ne 0$, 
$\psi_k'=0$ is equivalent to ${N}/{D} = {N'}/{D'}$, and 
$\psi_k'=\psi_k''=0$ is equivalent to ${N}/{D} = {N'}/{D'}
={N''}/{D''}$. This follows from the quotient rule:
\begin{eqnarray} 
\psi' & = & \frac{DN'-ND'}{D^2}, \cr 
\psi'' & = & \frac{DN'' - ND''}{D^2} -2\frac{D'(DN'-ND')}{D^3}.
\end{eqnarray}

\paragraph{Step 1.} Since $N$ and $D$ have double roots at $\tilde \E = \E$,
we can do a Taylor series for both of them near $\tilde \E = \E$:
\begin{eqnarray}
 \psi_k(\E,\tilde \E) &=& \frac{N''(\E,\E)(\tilde \E - \E)^2/2 + N'''(\E,\E)(\tilde \E-\E)^3/6 + \cdots}{D''(\E,\E)(\tilde \E - \E)^2/2 + D'''(\E,\E)(\tilde \E-\E)^3/6 + \cdots}
\cr
&=& \frac{N''(\E,\E) + N'''(\E,\E)(\tilde \E-\E)/3 + \cdots}
{D''(\E,\E) + D'''(\E,\E)(\tilde \E-\E)/3 + \cdots}.
\end{eqnarray}
$\psi_k'(\E,\E)=0$ is then equivalent to 
\begin{eqnarray}
N''(\E,\E)D'''(\E,\E) &=& N'''(\E,\E)D''(\E,\E) \cr && \cr 
\frac{-k(k-1)(k-2)\E^{k-3}}{\E(1-\E)} & = & \frac{-k(k-1)\E^{k-2}(1-2\E)}{\E^2(1-\E)^2}\cr && \cr 
(k-2)(1-\E) & = & 1-2\E \cr 
k \E & = & k-1.
\end{eqnarray}

\paragraph{Step 2.} If $\psi_k'=\psi_k''=0$, then we must have 
$N'D''=D'N''$ and $ND''=DN''$. We will explore these in turn.  
We write
\be 0 = N'D'' - D'N'' = \int_{\tilde\E}^\E D''(\E,\tilde \E) \dot N'(x,\tilde \E) 
- N''(\E,\tilde \E) \dot D'(x,\tilde \E) dx.\ee
Explicitly, this becomes
\be\label{signedmass} 0 = \int_{\tilde \E}^\E \frac{k(k-1)}{\tilde \E(1-\tilde \E)x(1-x)}
\left [\tilde \E^{k-1}(1-\tilde \E) - x^{k-1}(1-x) \right ] dx.\ee

The function $x^{k-1}(1-x)$ has a single maximum at $x=(k-1)/k$. If both
$\E$ and $\tilde \E$ are on the same side of this maximum, then the integrand
will have the same sign for all $x$ between $\tilde \E$ and $\E$, and the 
integral will not be zero. Thus we must have $\E < (k-1)/k < \tilde \E$, 
or vice-versa, and we must have $\E^{k-1}(1-\E) < \tilde \E^{k-1}(1-\tilde \E)$.  
Note that in this case the integrand changes sign exactly once. 
 
Now we apply the same sort of analysis to the other equation:
\be 0 = ND'' - DN'' = \int_{\tilde \E}^\E D''(x,\tilde \E) \dot N(x,\tilde \E) 
- N''(x,\tilde \E) \dot D(x,\tilde \E) dx.\ee
Explicitly, this becomes
\be 0 = \int_{\tilde \E}^\E \frac{k(k-1)}{\tilde \E(1-\tilde \E)x(1-x)}
\left [\tilde \E^{k-1}(1-\tilde \E) - x^{k-1}(1-x) \right ](\tilde \E - x) dx.\ee
This is the same integral as before, only with an extra factor of 
$(\tilde \E - x)$. If we view the first integral \eqref{signedmass} as a mass distribution
(with total mass zero), then the second integral is  (minus) the first 
moment of this mass distribution relative to the endpoint $\tilde \E$. 
But we have already seen that the 
distribution changes sign exactly once, and so 
must have a non-zero first moment. 
This is a contradiction. 

\paragraph{Step 3.} If $ND'=DN'$, then $N/D = N'/D'$. Call this
common ration $r$. Then 
\be
N = rD \quad \hbox{ and } \quad  N'  =  rD'.
\ee
Note that $N'$ and $D'$
are odd under interchange of $\E$ and $\tilde \E$, so the second
equation is invariant under this interchange.  Furthermore, we
have $(\tilde \E-\E)N' -N = r [ (\tilde \E - \E)D' - D]$. However,
$(\tilde \E - \E)N' - N$ is the same as $N$ with the roles of $\E$ and
$\tilde \E$ reversed, while $(\tilde \E - \E)D' - D$ is the same as
$D$ with the roles of $\E$ and $\tilde \E$ reversed. Thus the two equations
are satisfied for $(\E, \tilde \E)$ if and only if they are satisfied for
$(\tilde \E, \E)$. 

\paragraph{Step 4.} For $k=2$ we explicitly compute that $\psi_2'=0$ only at $\tilde \E = 1-\E$. 
If $k_{min}$
is the infimum of all values of $k$ for which $\psi_k$ has multiple critical
points, then at a critical point of $\psi_{k_{min}}$ we must
have $\psi_k'=\psi_k''=0$, which is a contradiction. Thus $k_{min}$ 
does not exist,
and $\psi_k$ has a unique critical point for all $k \ge 2$. In particular, 
$\zeta_k$ is a well-defined function.

\paragraph{Step 5.} The function $\zeta_k$ is defined by the condition
that $D N' - N D' = 0$ (and $\tilde \E \ne \E$, except when $\E = (k-1)/k$).
Let $f(\tilde e, e) = DN' - ND' = D^2 \psi'$. Moving along the curve
$\tilde \E = \zeta_k(\E)$ (that is, $f=0$), we differentiate implicitly:
\be
 0 = df = \dot f d\E + f' d\tilde \E,
\ee
so 
\be
 \frac{d \tilde \E}{d \E} = \frac{-\dot f}{f'}.
\ee

We compute $f' = D N'' - N D''.$ This is nonzero by Step 2. We also have
\begin{eqnarray} \dot f & = & D \dot N' - \dot N D' + \dot D N' - N \dot D' \cr
& = & -2S_0''(\E) (D - (\tilde \E-\E)D)' + k(k-1)\E^{k-2}(N - (\tilde \E-\E)N') \cr
& = & 2 S_0''(\E) [\E^k-\tilde \E^k + k(\tilde \E-\E)\tilde \E^{k-1}]
-2k(k-1)\E^{k-2}[S_0(\E)-S_0(\tilde \E)+(\tilde \E-\E)S_0'(\tilde \E)] \cr 
& = & D(\tilde \E,\E)N''(\tilde \E,\E) - N(\tilde \E,\E)D''(\tilde \E,\E).
\end{eqnarray}
The arguments in the last line are written in the correct order!  That
is, $\dot f$ is the same as $f'$, only with the roles of $\E$ and
$\tilde \E$ reversed. Since the equation $f=0$ is symmetric in $\E$
and $\tilde \E$, the argument of Step 2 can be repeated to show that
$\dot f \ne 0$.

Since $d\tilde \E/d\E$ is never zero, and since $d \tilde \E/d\E=-1$
at the fixed point (by symmetry), $\zeta_k'(\E) = d\tilde \E/d\E$ must
always be negative.
\end{proof}

\section{Theorem {\ref{Main-Thm-Simple}} for $k$-stars}
\label{SEC:k-stars}

\begin{theorem}\label{thm1}
Let $H$ be a $k$-star and suppose that $\E \ne (k-1)/k$.  
Then there exists a number $\T_0 > \E^k$ such that for all $\T
\in (\E^k ,\T_0)$, the entropy-optimizing graphon at $(\E,\T)$ is
unique and bipodal.  The parameters $(c, p_{11}, p_{12}, p_{22})$ are analytic functions of $\E$ and $\T$. 
As $\T$ approaches $\E^k$ from above, $p_{22} \to \E$, $p_{12} \to \zeta_k(\E)$, $p_{11}$ satisfies 
$S_0'(p_{11}) = 2S_0'(p_{12}) - S_0'(p_{22})$ and $c=O(\T-\E^k)$.
\end{theorem}

\begin{proof}
The entropy-maximizing graphon for each $(\E,\T)$ is multipodal \cite{KRRS}, and the parameters $\{c_j\}$ and $\{p_{12}\}$
must satisfy the optimality conditions \eqref{cmeq0}, \eqref{cmeq}. 
The first step of the proof is to estimate the terms in the optimality
equations to within $o(1)$. This will determine the solutions to
within $o(1)$ and demonstrate that our optimizing graphon is close to
bipodal of the desired form. The second step, based on a separate
argument, will show that the optimizer is exactly bipodal. 
The third step shows that the optimizer is in fact unique.

In doing our asymptotic analysis, our small parameter is $\Delta \T :=
\T - \E^k$.  But we could just as well use $\Delta s := s(g) -
S_0(\E)$ or the squared $L^2$ norm of $\Delta g := g - g_0$, where
$g_0(x,y) = \E$ (here $g$ denotes the graphon as a function $[0,1]^2
\to [0,1]$, not a vector of multipodal parameters.) {We claim that} these are all of
the same order. Through 
arguments found in
\cite{RS2}, one can bound $\Delta \T$ above by a multiple of $\|
\Delta g\|^2$, and bound $|\Delta s|$ below by a multiple of $\|
\Delta g \|^2$. By considering a bipodal graphon with
$p_{11}=p_{12}=\zeta_k(\E)$ and $p_{22}$ close to $\E$, we can bound
$|\Delta s|$ above by a constant times $\Delta \T$. This shows that
$O(\Delta s) = O(\Delta \T)$, and $O(\| \Delta g\|^2)$ is trapped in
between.

Order the clusters so that the largest cluster is the last cluster {(of length $c_M$)}.  By subtracting the 
equation \eqref{cmeq} 
for $c_M$ from 
the equations for $c_j$, we eliminate $\gamma$ from our equations:
\begin{eqnarray}\label{KKT1}
S_0'(p_{ij}) & = & \alpha + \frac{k}{2} \beta(d_i^{k-1} + d_j^{k-1}) \cr 
2 \sum_{j=1}^M c_j \left ( S_0(p_{ij}) \!-\! S_0(p_{Mj}) \right ) &=& 2 \alpha 
(d_i\!-\!d_M) + \beta
\left ( d_i^k \!-\! d_M^k \!+\! k \sum_{j=1}^M c_j d_j^{k-1} (p_{ij}\!-\!p_{Mj})\right ).
\end{eqnarray}

\paragraph{Step 1.} Since $\|\Delta g\|$ is small, the area of the
region where $g(x,y)$ differs substantially from $\E$ must be
small. Thus all clusters must either have $d_i$ close to $\E$ or $c_i$
close to zero (or both). We call a cluster Type I if $c_i$ is close to
0 and Type II if $d_i$ is close to $\E$. (If a cluster meets
both conditions, we arbitrarily throw it into one camp of the
other). The first equation in (\ref{KKT1}) implies that, for fixed
$i$, the values of $p_{ij}$ are nearly constant for all $j$ of Type
II. Since the $c_j$'s are small for $j$ of Type I, this common value
must be close to $d_i$. To within $o(1)$, our equations then simplify
to
\begin{eqnarray}\label{ELintegrated}
  S_0'(d_i) & = & \alpha + \frac{k}{2}\beta(d_i^{k-1}+ \E^{k-1}), \cr 
  S_0(d_i)-S_0(\E) & = & \alpha (d_i - \E) + \beta [d_i^k - \E^k + k\E^{k-1}(d_i-\E)].
\end{eqnarray}
Since $d_M = \E + o(1)$, the first of those equations applied to
$d_M$ implies that
\begin{equation}
 \alpha + k \E^{k-1} \beta = S_0'(\E) + o(1).
\end{equation}
We can thus replace $\alpha$ with $S_0'(\E) - k\E^{k-1} \beta + o(1)$ throughout. 
This gives the equations (again with $o(1)$ errors):
\begin{eqnarray}\label{ELsimple}
2(S_0'(d_i)-S_0'(\E)) & = & k\beta (d_i^{k-1} - \E^{k-1}), \cr 
2 [S_0(d_i)-S_0(\E) - S_0'(\E) (d_i-\E)] & = & 
\beta [d_i^k - \E^k - k \E^{k-1} (d_i -\E)].
\end{eqnarray}

There are two solutions to these equations.  One is simply to have 
$d_i=\E$, in which case both equations say $0=0$. Indeed, we already know that there must
be clusters with $d_i$ close to $\E$. In looking for solutions with
$d_i \ne \E$, the second equation says that 
$\beta = \psi_k(\E,d_i)$.  

We can also divide the first equation by the second to eliminate $\beta$.  This gives an equation that is algebraically equivalent to $\partial \psi_k(\E,d_i)/ \partial d_i=0$. 
In other words, $d_i$ must be the unique critical point {$\zeta_k(\E)$}
of $\psi_k$, and $\beta$ must be the critical value.  
In fact, the critical point is a maximum of $\psi_k$.  Remember that $s(\E,\T) = s(\E,\E^k) + \int_{\E^k}^{\T} \beta$ from (\ref{ftc}).  
Since the computation of $\beta$ is
independent of $\Delta \T$ (to lowest order), we have $s(\E,\T)-s(\E,\E^k)
= \beta \delta \T + o(\Delta \T)$, so maximizing $\beta$ is tantamount
to maximizing $s$. 
 
\paragraph{Step 2.} We have shown so far that
the optimizing graphon is multipodal, with all of the clusters either
having $d_i$ close to $\zeta_k(\E)$ or close to $\E$. We refine our
definitions of Type I and Type II so that all the clusters with
$d_i$ close to $\zeta_k(\E)$ are Type I and all the clusters
with $d_i$ close to $\E$ are Type II.  Since the value of
$g(x,y)$ is determined by $d(x)$ and $d(y)$ (and $\alpha$ and
$\beta$), this means that the optimizing graphon is nearly constant
(i.e. with pointwise small fluctuations) on each quadrant. We order the 
clusters so that the Type I clusters come before Type II. 

Let $g_b$ be the bipodal graphon obtained by averaging over each
quadrant.  Let $\Delta g_f = g-g_b$. (The f stands for ``further''.) 
We will show
that having $\Delta g_f$ non-zero is an inefficient way to increase
$\T$, that is, $(s(g)-s(g_b))/(\T(g)-\T(g_b))$ is less than $\beta$. {This will imply that $\Delta g_f =0$ and
so $g=g_b$.}

Since $\T = \int_0^1 d(x)^k dx$, the changes in $\T$ are a function only of the marginal distributions of $\Delta g_f$. Once these are fixed, the values of $\Delta g_f$ on each quadrant must take the form
\begin{equation}\label{maxsum}
\Delta g_f(x,y) = \hbox{(function of $x$)} + \hbox{(function of
  $y$)}. 
\end{equation}
The reason is that we can write the entropy on each quadrant as $\iint
S_0(g_b + \Delta g_f) = \iint S_0(g_b) + S_0'(g_b) \Delta g_f +
(1/2) S_0''(g_b) \Delta g_f^2 + \cdots$. The first term is
independent of $\Delta g_f$ and the second is zero (since $g_b$ was
assumed to equal the average value of $g = g_b + \Delta g_f$ on the
quadrant). Since the changes to the graphon are {\em pointwise} small,
we can ignore terms past the second, so we are basically left with
$S_0''(g_b)/2$ times the squared $L^2$ norm of $\Delta g_f$ on the
quadrant, which we then minimize subject to the constraint that the
marginal distributions are fixed.
We can write $\Delta g_f(x,y) = \phi_1(x) + \phi_2(y) +\phi_3(x,y)$,
where $\phi_1$ and $\phi_2$ give the two fixed marginals, and $\phi_3$
has zero marginals.  But then $\int \Delta g_f^2 = \int \phi_1^2 +
\phi_2^2 + \phi_3^2$, since all of the cross terms integrate to
zero. (Integrating $\phi_2(y)\phi_3(x,y)$ over $x$ or
$\phi_1(x)\phi_3(x,y)$ over $y$ gives zero since $\phi_3$ has zero
marginals, and integrating $\phi_1(x)\phi_2(y)$ over either $x$ or $y$
gives zero since $\phi_1$ and $\phi_2$ have mean zero). The way to
minimize $\int \Delta g_f^2$ is simply to take $\phi_3=0$. {This establishes \eqref{maxsum}}.

Furthermore, to maximize $\T(g)-\T(g_b)$, the functions of $x$ should
be the same (up to scale) in the $I$-$I$ and $I$-$II$ quadrants, and
the same (up to scale) in the $II$-$I$ and $II$-$II$ quadrants. This
is because $\T(g)-\T(g_b) \approx \int_0^1 k(k-1) \E^{k-2} \delta
d(x)^2 dx$ involves a cross term between the contributions to $\delta
d(x)$ from two quadrants, and this cross term is maximized when the
corresponding functions point in the same direction.

The upshot is that there are functions $F_1(x)$ on $[0,c]$ and
$F_2(x)$ on $[c,1]$, each with mean zero and normalized to have
root-mean-squared 1, and constants $\mu$, $\nu$, $\kappa$, $\lambda$,
such that 
\vskip1truein
\begin{eqnarray} \Delta g_f(x,y) & = & \mu F_1(x) + \mu
  F_1(y) \hbox{ on the $I$-$I$ square.} \cr \Delta g_f(x,y) & = & \nu
  F_1(x) + \kappa F_2(y) \hbox{ on the $I$-$II$ rectangle.} \cr \Delta
  g_f(x,y) & = & \lambda F_2(x) + \lambda F_2(y) \hbox{ on the
    $II$-$II$ square.}
\end{eqnarray}

Now we compute the changes in $\T$ and in $s$, to second order in
$(\mu, \nu, \kappa, \lambda)$, noting that all of the first-order
changes are zero, and that the integral of $\Delta g_f^2$ over the
$I$-$I$ square, the two rectangles, and the $II$-$II$ square are $2c^2
\mu^2$, $2c(1-c)(\nu^2+\kappa^2)$, and $2(1-c)^2 \lambda^2$,
respectively.

 \begin{eqnarray} 
 s(g)-s(g_b) & = & \mu^2 c^2 S_0''(p_{11}) 
+ \nu^2 c(1-c) S_0''(p_{12}) \cr
&+& \kappa^2 c(1-c) S_0''(p_{12}) + \lambda^2(1-c)^2 S_0''(p_{22}),\cr
\T(g)-\T(g_b) & = & c k(k-1) d_1^{k-2}(\mu c + \nu(1-c))^2/2\ +\cr 
&+& (1-c)k(k-1)d_2^{k-2}(\kappa c + \lambda(1-c))^2/2.
 \end{eqnarray}
 Both the change in $s$ and the change in $\T$ are the sum of two
 terms, one involving $\mu$ and $\nu$,
 and the other involving $\kappa$ and $\lambda$. Let:
 
\begin{eqnarray} 
A_1 &=&  \mu^2 c^2 S_0''(p_{11}) + \nu^2 c(1-c) S_0''(p_{12}),\cr   
A_2 &=&  \kappa^2 c(1-c) S_0''(p_{12}) + \lambda^2(1-c)^2 S_0''(p_{22}),\cr
B_1 &=& ck(k-1)d_1^{k-2} (\mu c + \nu(1-c))^2/2,\cr 
B_2 &=& (1-c)k(k-1)d_2^{k-2}(\kappa c + \lambda(1-c))^2/2, 
\end{eqnarray}
so to lowest order, 
\be
\frac{S(g)-S(g_b)}{\T_k(g)-\T_k(g_b)} =  \frac{A_1+A_2}{B_1+B_2}.
\ee

For the perturbations involving only $\kappa$ and $\lambda$, the
 ratio $A_2/B_2$ depends only on $r=\kappa/\lambda$:
\begin{equation}
 \frac{A_2}{B_2} = \frac{2 [r^2 cS_0''(p_{12}) 
  + (1-c)S_0''(p_{22})]}{k(k-1)d_2^{k-2}(rc + (1-c))^2/2}.
\end{equation}
We optimize by taking a
derivative w.r.t.~$r$ and setting it equal to zero, with the result
that $r = S_0''(p_{22})/S_0''(p_{12})$, independent of $c$. Since $r$
does not diverge as $c \to 0$, the limit of ${A_2}/{B_2}$ as $c \to 0$ can 
be obtained by simply setting $c=0$, giving a
limiting ratio of $2 S_0''(\E)/[k(k-1)d_2^{k-2}] = \psi_k(\E,\E) < \beta$. 
Since the
limit is less than $\beta$, the ratio must be smaller than $\beta$ for
all sufficiently small values of $c$.  

Almost identical arguments apply to the perturbations involving only $\mu$
and $\nu$. The optimal ratio $\mu/\nu$ is then
$S_0''(p_{12})/S_0''(p_{11})$, which again cannot diverge as $c \to
0$. Thus for small values of $c$ the dominant terms
are those involving $\nu$, and the ratio $A_1/B_1$
approaches $2 S_0''(p_{12})/[k(k-1)d_1^{k-2}]$. 
But $d_1 \approx p_{12} \approx \tilde \E$, so our
ratio goes to $2 S_0''(\tilde \E)/[k(k-1)\tilde \E^{k-2}] 
= \psi_k(\tilde \E, \tilde \E) < \beta$. 

Thus there is a constant $\beta_0<\beta$ such that $A_1 \le \beta_0
B_1$ and $A_2 \le \beta_0 B_2$, so $A_1+A_2 < \beta_0 (B_1 + B_2)$, so
\be 
s(g)-s(g_b) \le \beta_0 (\T(g)-\T(g_b)).  
\ee 

{However $ds/d\T
\approx \beta$ for changes in $c$ that preserve the bipodal structure. This means if we perturb a bipodal graphon
to maximize $s$, it is better to perturb $c$ than to make $(\mu,\nu,\kappa,\lambda)$ nonzero.}
Thus $\kappa$ and $\lambda$ must both be zero,
implying that there is only one Type II cluster, and $\mu$ and $\nu$
must be zero, implying that there is only one Type I cluster.

\paragraph{Step 3.}

We have established that the minimizing graphon is bipodal, with
$p_{22} \approx \E$ and $p_{12} \approx \zeta_k(\E)$ .  We now show
that the form of this graphon is unique.  Since the equation is
bipodal, we consider the exact optimality equations. After eliminating
$\gamma$, we have
\begin{eqnarray} 
  S_0'(p_{11}) & = &  \alpha + k \beta d_1^{k-1}, \cr
  S_0'(p_{12}) & = & \alpha + \frac{k}{2} \beta(d_1^{k-1} + d_2^{k-1}), \cr 
  S_0'(p_{22}) & = & \alpha + k \beta d_2^{k-1}, \cr
  \frac{\partial S}{\partial c} & = & \alpha \frac{\partial \E}{\partial c} + \beta \frac{\partial \T}{\partial c}, \cr 
  \E & = & \E_0, \cr 
  \T & = & \T_0.
\end{eqnarray}
We use the second and third equations to solve 
for $\alpha$ and $\beta$: 
\begin{eqnarray} \label{alphabeta-k}
\alpha & = & \frac{-S_0'(p_{22}) (d_2^{k-1}+d_1^{k-1}) + 2d_2^{k-1}S_0'(p_{12})}
{d_2^{k-1} - d_1^{k-1}}, \cr 
&& \cr 
\beta & = & \frac{2}{k} \frac{S_0'(p_{22}) - S_0'(p_{12})}{d_2^{k-1}-d_1^{k-1}}.
\end{eqnarray}
 Plugging this into the first equation then gives 
 \begin{equation}
 S_0'(p_{11}) - 2 S_0'(p_{12}) + S_0'(p_{22}) = 0. 
\end{equation}
This leaves four equations in four unknowns, which we write as 
\begin{equation}
\vec f = \begin{pmatrix} 0 \cr 0 \cr e_0 \cr t_0
\end{pmatrix},
\end{equation}
where 
\begin{eqnarray}
f_1 & = &  S_0'(p_{11}) - 2 S_0'(p_{12}) + S_0'(p_{22}), \cr 
f_2 & = & \frac{\partial s}{\partial c} - \alpha \frac{\partial \E}{\partial c} - \beta \frac{\partial \T}{\partial c}, \cr 
f_3 & = & c^2 p_{11} + 2 c(1-c) p_{12} + (1-c)^2 p_{22}, \cr 
f_4 & = & c d_1^k + (1-c) d_2^k, 
\end{eqnarray}
and where $\alpha$ and $\beta$ are given by (\ref{alphabeta-k}). 

We know a solution when $\T_0 = \E_0^k$, namely $p_{22}=\E_0$, $p_{12} =
\zeta_k(\E_0)$, $c=0$ and $p_{11} = S_0'{}^{-1}(2S_0'[\zeta_k(\E_0)] -
S_0'(\E_0))$. We will show that $d\vec f$ has non-zero determinant at
this point.  By the inverse function theorem, this implies that, when
$\T_0$ is close to $\E_0^k$, there is only one
value of $(p_{11},p_{12}, p_{22}, c)$ close to this point for which
$f(p_{11},p_{12}, p_{22}, c) = (0,0,\E_0, \T_0)^T$.  Moreover, the
parameters $(p_{11}, p_{12}, p_{22}, c)$ depend analytically on $\E_0$
and $\T_0$. This will complete the proof. (Note that we have reordered the variables by listing $c$ last.) 

The derivatives of $f_1$, $f_3$, and $f_4$ are: 
\begin{eqnarray}
df_1 & = & (S_0''(p_{11}), -2 S_0''(p_{12}), S_0''(p_{22}), 0), \cr 
df_3 & = & (c^2, 2c(1-c), (1-c)^2, 2cp_{11} + 2(1-2c)p_{12} -2(1-c)p_{22}), \cr 
df_4 & = & (k c^2 d_1^{k-1}, kc(1-c) (d_1^{k-1}+d_2^{k-1}), k(1-c)^2 d_2^{k-1}, \cr 
&&  d_1^k -d_2^k + kcd_1^{k-1}(p_{11}-p_{12}) + k(1-c) d_2^{k-1}(p_{12}-p_{22})).
\end{eqnarray}
Evaluating at $c=0$ gives
\begin{eqnarray}
df_1 & = &  (S_0''(p_{11}), -2 S_0''(p_{12}), S_0''(p_{22}), 0), \cr 
df_3 & = & (0,0,1, 2p_{12} -2p_{22}), \cr 
df_4 & = & (0,0, kp_{22}^{k-1}, p_{12}^k-p_{22}^k + kp_{22}^{k-1}(p_{12}-p_{22})). 
\end{eqnarray}
$df$ is block triangular, with $2 \times 2$ blocks. The
lower right block has determinant $p_{12}^k -p_{22}^k -
kp_{22}^{k-1}(p_{12}-p_{22}) = D(p_{22},p_{12})$, which is 
non-zero when $p_{12} \ne p_{22}$, i.e. when $\E_0 \ne (k-1)/k$.

Also  $\partial f_2/\partial p_{11} = 0$ when $c=0$, since $\alpha$
and $\beta$ are independent of $p_{11}$ (when $c=0$) and since $\partial^2
S/\partial c \partial p_{11}$, $\partial^2 \E/\partial c \partial
p_{11}$ and $\partial^2 t/\partial c \partial p_{11}$ are all $O(c)$.
As a result, 
\be \det(df) = S_0''(p_{11}) {\partial f_2}/{\partial p_{12}}
D(p_{22},p_{12}). \ee

So as long as $p_{12} \ne p_{22}$ (i.e. as long as $\E_0 \ne (k-1)/k$),
everything boils down to computing $\partial f_2/\partial p_{12}$ at
$c=0$ and seeing that it is nonzero.  
We compute
\begin{eqnarray} 
\frac{\partial \beta}{\partial p_{12}} & = & \frac{2}{k} 
\frac{(p_{22}^{k-1}-p_{12}^{k-1})(-S_0''(p_{12}))-(S_0'(p_{22})-S_0'(p_{12}))(-(k-1)p_{12}^{k-2})}{(p_{22}^{k-1}-p_{12}^{k-1})^2} \cr 
& = &  \frac{2}{k} \frac{(k-1)p_{12}^{k-2}(S_0'(p_{22})-S_0'(p_{12})) - (p_{22}^{k-1}
-p_{12}^{k-1})S_0''(p_{12})}{(p_{22}^{k-1}-p_{12}^{k-1})^2} 
\end{eqnarray}
at $c = 0$.  We will show separately that this quantity is nonzero. 

Since $\alpha = S_0'(p_{22}) - k \beta d_2^{k-1}$,
\begin{eqnarray} \frac{\partial \alpha}{\partial p_{12}} & = & -k
  d_2^{k-1} \frac{\partial \beta}{\partial p_{12}} -k(k-1)\beta
  d_2^{k-2} \frac{\partial d_2}{\partial p_{12}} \cr 
& = & -k
  d_2^{k-1} \frac{\partial \beta}{\partial p_{12}} - k(k-1) d_2^{k-2}
  c \beta \Rightarrow -k p_{22}^{k-1} \frac{\partial \beta}{\partial
    p_{12}},
\end{eqnarray}
where $\Rightarrow$ denotes a limit as $c \to 0$. 
We also compute
\begin{eqnarray} \frac{\partial^2 S}{\partial c \partial p_{12}} 
  & =  & 2(1-2c) S_0'(p_{12}) \Rightarrow 2S_0'(p_{12}) \cr \cr
  \frac{\partial^2 e}{\partial c \partial p_{12}} & = & 2(1-2c)
  \Rightarrow 2 \cr \cr 
\frac{\partial^2 t}{\partial c \partial
    p_{12}} & = & k(1-2c)(d_1^{k-1}+d_2^{k-1}) \Rightarrow k
  (p_{12}^{k-1} + p_{22}^{k-1})
   \end{eqnarray}
Finally we combine everything:
\begin{eqnarray} \frac{\partial f_2}{\partial p_{12}}\Big |_{c=0} & = &
  \frac{\partial^2 S}{\partial c \partial p_{12}} - \frac{\partial
    \alpha}{\partial p_{12}} \frac{\partial \E}{\partial c} - \alpha
  \frac{\partial^2\E}{\partial c \partial p_{12}} - \frac{\partial
    \beta}{\partial p_{12}} \frac{\partial \T}{\partial c} - \beta
  \frac{\partial^2\T}{\partial c \partial p_{12}} \cr \cr & = &
  2S_0'(p_{12}) - 2 \alpha - \beta k(p_{12}^{k-1} + p_{22}^{k-1}) \cr && +
  \left (kp_{22}^{k-1}(2 p_{12}-2 p_{22}) - (p_{12}^k-p_{22}^k + k
  p_{22}^{k-1}(p_{12}-p_{22})) \right )\frac{\partial \beta}{\partial p_{12}}.
   \end{eqnarray}
The terms not involving $\partial \beta/\partial p_{12}$ all cancel, by the
second variational equation, and we are left with 
\begin{equation}
 \frac{\partial f_2}{\partial p_{12}} = -D(p_{12},p_{22})
\frac{\partial \beta}{\partial p_{12}}.
\end{equation}
 
Finally, we need to show that $\partial \beta/\partial p_{12} \ne 0$.
Since $p_{12}$ maximizes $\psi_k(p_{22},p_{12})$ (for fixed $p_{22}$), we must have 
{(referring to the notation of the proof of Theorem \ref{thm-zeta})}
$(N/D)'=0$, or equivalently $N'/D' = N/D$, where we write $\psi_k =
N/D$, as above.  But $\beta = N'/D'$. If $\partial \beta/\partial
p_{12}$ were equal to zero, then we would have $N''/D'' = N'/D'$. But
we have previously shown that it is impossible to simultaneously have
$N/D = N'/D' = N''/D''$, except at $p_{12} = p_{22} = (k-1)/k$, so
$\partial \beta/\partial p_{12}$ must be nonzero whenever $\E_0 \ne
(k-1)/k$.  This makes $\det(d\vec f)$ nonzero at
$(p_{11},\zeta_k(\E_0), \E_0,0)$, so the solutions near this point are
unique and analytic in $(\E,\T)$.
\end{proof}

\section{Theorem \protect{\ref{Main-Thm-Simple}} for $k$-starlike graphs.}
\label{SEC:k-simple}

Now suppose that $H$ is a $k$-starlike graph with $\ell$ edges, and with 
$n_k$ vertices of degree $k$, and let 
$\T$ be the density of $H$ and $\T_k$ be the density of $k$-stars.  Our first result relates 
$\Delta \T := \T - \E^\ell$ to  $\Delta \T_k := \T_k-\E^k$. 

\begin{lemma}\label{LEM:close} 
If $g$ is an entropy-maximizing graphon for $(\E,\T)$ with
$\T > \E^\ell$, then $\Delta \T  = n_k \E^{\ell-k} \Delta \T_k + 
O(\Delta \T_k^{3/2})$. 
\end{lemma}

\begin{proof} 
Writing $g(x,y) = \E + \Delta g(x,y)$, we
expand $\T$ as a polynomial in $\Delta g$:
\be
\T = \int d{\bf x} \prod g(x_i, x_j)= \int d{\bf x} \prod (\E
+ \Delta g(x_i, x_j)), 
\ee
where there is a variable $x_i$ for each vertex of $H$ and the product is
over all edges in $H$. 
 
The 0-th order term is $\E^\ell$. 
The first-order term is identically zero, since 
$\iint \Delta g(x,y) dx\, dy = \Delta \E = 0$. When looking at
higher-order expansions, there are some terms that come from having
{all} $\Delta g$'s along edges that share a single vertex of degree $k$. {These terms also appear in the expansion of $\T_k$,} so the sum of those terms is exactly
$\E^{\ell -k} \Delta \T_k$. Since all vertices have degree $k$ or 1, 
summing these terms gives $n_k \E^{\ell-k} \Delta \T_k$. 

What remains are terms where the $\Delta g$'s refer to edges that do not
all share a vertex.
We bound these in turn. In each case, let $\{ e_i \}$
be the set of edges that correspond to factors of $\Delta g$. 
\begin{itemize}
\item If one of the $e_i$'s is disconnected from the rest, then the integral
is exactly zero. So we can assume that all connected components of 
$\{ e_i \}$ contain at least two edges. 

\item If there is more than one connected component, then we get a product of 
factors, one for each connected component. Each factor is bounded by a constant
times $\| \Delta g \|^2$, so the product is $O(\| \Delta g \|^4)$. 

\item If there is only one connected component, whose edges do not all share
a vertex, then $\{ e_i \}$ either contains a triangle or a chain of three
consecutive edges. We bound such terms by taking absolute values of all
the $\Delta g$'s and setting all terms other than the three edges in the 
triangle or 3-chain to 1. The result is either a constant times 
$\iiiint |\Delta g(w,x)| |\Delta g(x,y)| |\Delta g(y,z)| dw\, dx \, dy \, dz$,
or by a constant times $\iiint |\Delta g(x,y)| |\Delta g(y,z)| |\Delta g(z,x)|
dx \, dy \, dz$, either of 
which in turn is bounded by a constant times $\| \Delta g \|^3$. 
(If we then think of $|\Delta g|$
as the integral kernel of an operator $L$ on $L^2(0,1)$, then the integral for
a 3-chain is the expectation of $L^3$ in a particular state, and the integral
for a triangle is the trace of $L^3$. Both are bounded by $Tr(L^2)^{3/2}=
\| |\Delta g| \|^3 = \| \Delta g \|^3$. )
\end{itemize}
Since $\Delta \T_k$ scales as $\| \Delta g \|^2$, all the corrections to 
the approximation $\Delta \T \approx n_k \E^{\ell-k} \Delta \T_k$ are 
$O(\Delta \T_k^{3/2})$ or smaller. 
\end{proof}

\subsection{Proof of Theorem \protect{\ref{Main-Thm-Simple}}}

Since $\Delta \T$ is proportional to $\Delta \T_k$ (plus small errors), 
the problem of optimizing $\Delta s/\Delta \T$ is a small perturbation of 
the problem of optimizing $\Delta s/ \Delta \T_k$, or equivalently 
optimizing $\Delta s$ for fixed $\Delta \T_k$, which we solved in the last
section. Since that problem has a unique optimizer, any optimizer for
$\Delta s/\Delta \T$ must come close to optimizing $\Delta s/\Delta \T_k$,
and so must be close to the bipodal graphon derived in Theorem
\ref{thm1}.

We can thus write $g =
g_b + \Delta g_f$, as in the last steps of the proof of Theorem
\ref{thm1}, where $g_b = \E + \Delta g_b$ 
is a bipodal graphon with $p_{22} \approx \E$ and 
$p_{12} \approx \zeta_k(\E)$ and where 
$\Delta g_f$ is a function that averages to zero on each quadrant of $g_b$.

\begin{lemma} The function $\Delta g_f$ is pointwise small. That is, as
$\T \to \E^\ell$, $\Delta g_f$ goes to zero in sup-norm. 
\end{lemma}

\begin{proof}[Proof of lemma]

Since we no longer in the setting where the entropy maximizer is proven to
be multipodal, we cannot use the equations (\ref{KKT1}) directly. 
However, we can still apply the method of 
Lagrange multipliers to pointwise variations of the graphon. (See
\cite{KRRS} for a rigorous justification.) 
These variational equations are 
\be \label{EL-k-simple} \frac{1}{2} \ln \left ( \frac{1}{g(x,y)} - 1 \right ) = 
\frac{\delta s}{\delta g(x,y)} = \alpha + 
\beta \frac{\delta \T}{\delta g(x,y)}.
\ee
We need to compute $\delta \T/\delta g$ and show that it is nearly
constant on each quadrant. Since $\alpha$ and $\beta$ are constants, this 
would imply that $g(x,y)$ is nearly constant on each quadrant, and hence that
$\Delta g_f$ is pointwise small. Let $g_0(x,y)=\E$. 

Since $\| \Delta g \|$
is small (where $\Delta g =g-g_0 = 
\Delta g_b + \Delta g_f$),
we can find a small constant $a=o(1)$
such that, for all $x$ outside a set $U\subset[0,1]$ of measure $a$, $\int_0^1
|\Delta g(x,y)| dy < a$.  (This set $U$ is essentially what we
previously called the Type I clusters, but at this stage of the
argument we are not assuming a multipodal structure. Rather, we are 
just using the fact that $\T - e^\ell = O(\| \Delta g\|^2)$.)

The functional derivative $\delta \T/\delta g$ has a diagrammatic 
expansion similar to the expansion of $\T$. 
For each edge of $H$, we get a contribution by deleting the
edge, assigning the values $x$ and $y$ to the endpoints of the edge,
and integrating over the values of all other vertices.  Since $U$ is
small, we can estimate $\delta \T/\delta g$ to within $o(1)$ by
restricting the integral to $(U^c)^{v-2}$, where $v$ is the
number of vertices in $H$ and $U^c$ is the complement of $U$. 
This implies that terms involving $\Delta
g$ can only contribute non-negligibly on edges connected to $x$ or to
$y$. Furthermore, they can only contribute when attached to $x$ if $x
\in U$, and can only contribute when attached to $y$ if $y \in U$.

We now begin a bootstrap. We will show that $\delta \T/\delta g$ is nearly
constant on each quadrant {$U^c\times U^c,U\times U^c, U\times U$} in turn. This will show that $g$ is nearly constant
on that quadrant, which will help us prove that $\delta \T/\delta g$ is
nearly constant on the next quadrant. 

If $x$ and $y$ are both in $U^c$, then 
the contributions of the terms involving $\Delta g$ are negligible, so
$\delta \T/\delta g(x,y)$ can be computed, to within a small error,
using the approximation $g(x,y) \approx \E$. But when $g(x,y) = \E$, 
$\delta \T/\delta g(x,y)$ is independent of $x$ and $y$. Since 
$\delta \T/\delta g(x,y)$ is nearly constant on $U^c \times U^c$, 
equation (\ref{EL-k-simple}) implies that $g$ is nearly constant
on $U^c \times U^c$. In particular, $\Delta g_f$ is pointwise small on
$U^c \times U^c$.

Next suppose that $y \in U^c$ and $x \in U$. 
Then $\delta \T/\delta g(x,y)$ is nearly
independent of $y$, so $g(x,y)$ is nearly independent of $y$, and is
nearly equal to $d(x)$. But then the integrals involved in computing
$\delta \T/\delta g(x,y)$ are easy, where we use $g_0 + \Delta g$ on
the edges connected to $x$, $g_0$ on all other edges, and only
integrate over $(U^c)^{v-2}$. If the degree of $x$ is $k$, then the edges
connected to $x$ contribute 
$d(x)^{k-1} e^{\ell-k}$. Summing over edges, and symmetrizing over the assignment
of $x$ and $y$ to the two endpoints, we obtain the approximation
\be \label{grad-t}
\frac{\delta \T}{\delta g(x,y)} \approx \frac{k n_k \E^{\ell-k}}{2} 
\left (d(x)^{k-1} + d(y)^{k-1} \right ). 
\ee
Up to an overall factor of $n_k \E^{\ell-k}$, this is the same
functional derivative as for a $k$-star.  This also
applies if $x \in U^c$, except that in the
latter case $d(x) \approx \E$, and also applies if $x \in U^c$ and $y \in U$.

In other words, we can use the approximation (\ref{grad-t}) in 
(\ref{EL-k-simple}) whenever {\em either} $x$ or $y$ (or both) is in $U^c$.  
This implies that the integrated equations (\ref{ELintegrated}) 
apply for all $x$ (with $d_i$ replaced by $d(x)$, and with $\beta$ scaled up
by $n_k \E^{\ell-k}$). Following  the exact same reasoning as
in the proof of Theorem \ref{thm1}, we obtain 
that $d(x)$ only takes on 2 possible values
(up to $o(1)$ errors). We then define Type I and Type II points,
depending on whether the degree function is close to $\zeta_k(\E)$ or
$\E$, respectively, so that $U$ is precisely the set of Type I
points. Our graphon is then nearly constant on the $I$-$II$, $II-I$
and $II$-$II$ quadrants.

We still need to show that the graphon is nearly constant in the
$I$-$I$ quadrant. Suppose that $x$ and $y$ are in $U$.  In computing
$\delta \T/\delta g(x,y)$, we approximate our integral by integrating
over $(U^c)^{v-{2}}$. But if $z \in U^c$, then $g(x,z)$ is (nearly)
independent of $x$, since we have just established that $g$ is nearly
constant on the $I$-$II$ quadrant.  Thus $\delta \T/\delta g$ (which is 
obtained by integrating products of terms $g(x,z)$) is nearly
independent of $x$. Likewise, it is nearly independent of $y$, implying
that $g(x,y)$ is nearly constant on the $I$-$I$ quadrant.

Note, by the way, that the approximation (\ref{grad-t}) does not apply in the
$I$-$I$ quadrant; in that case $\delta \T/\delta g$ contains
terms with powers of both $d(x)$ and $d(y)$. However, that approximation
is not needed for our proof, since the $I$-$I$ quadrant only
contributes $O(c)$ to the integrated equations (\ref{ELintegrated}).
\end{proof}

Returning to the proof of Theorem \ref{Main-Thm-Simple}, 
we need to compare $s(g_b + \Delta g_f) - s(g_b)$ to $\T(g_b+\Delta
g_f)-\T(g_b)$.  

As before, we expand $\T(g)$ as the integral of a polynomial in $g$,
obtained by assigning $g_0 + \Delta g_b + \Delta g_f$ to each edge of
$H$ and integrating. The difference between $\T(g_b + \Delta g_f)$ and
$\T(g_b)$ consists of terms with at least one $\Delta g_f$.  However,
the terms with {\em exactly} one $\Delta g_f$ are identically zero,
since $g_b$ is constant on quadrants, and $\Delta g_f$ averages to
zero on each quadrant. Furthermore, terms for which all of the $\Delta
g_b$'s and $\Delta g_f$'s share a vertex are exactly what we would get
from the approximation $\Delta \T \approx n_k \E^{\ell-k}\T_k$.  Any
term  that distinguishes between $\Delta \T$ and $n_k \E^{\ell-k}
\Delta \T_k$ must have at least two $\Delta g_f$'s and either a third
$\Delta g_f$ or a $\Delta g_b$, forming either a 3-chain, a triangle,
or two connected $\Delta g_f$'s and a disconnected $\Delta g_b$.

Let $\Delta g_f'(x,y) = |\Delta g_f(x,y)|$, and let 
\be
\Delta g_b'(x,y)  = \begin{cases} 2c & x,y \in II, \cr 1 & \hbox{otherwise.}
\end{cases}\ee
 This is conveniently expressed in terms of outer products. 
Let $| 1 \rangle \in L^2([0,1])$ be the constant function 1, and let 
$|\omega \rangle$ be the function 
\be
\omega(x) = \begin{cases} 0 & x<c, \cr 
1 & x>c. \end{cases}\ee

 Then 
\begin{eqnarray}
  \Delta g_b' & = & | 1 \rangle \langle 1 | - |\omega \rangle \langle \omega | + 2c |\omega \rangle \langle \omega | \cr
  & = & |1 \rangle \langle 1-\omega| + |1 - \omega \rangle \langle \omega| + 2c |\omega \rangle \langle \omega|.
\end{eqnarray}

Note that $|\Delta g_b(x,y)| \le \Delta g_b'(x,y)$ for all $x,y \in (0,1)$. 
{To see this, the only issue is what happens when
$(x,y)$ is in the $II-II$ quadrant, since otherwise we trivially have 
$|\Delta g_b| \le 1$. Since $e(g)$ is fixed, $(1-c)^2$ times
$\Delta g_b(x,y)$ for $x,y > c$ equals minus the integral of $\Delta g_b$ 
over the other three quadrants. But the area of those three quadrants is 
$2c-c^2 < 2c$, and the biggest possible value of 
$|\Delta g_b|$ is $\max(e,1-e)<1$,
so $\frac{1}{(1-c)^2} \int |\Delta g_b|$ (integrated over the $I-I$, $I-II$
and $II-I$ quadrants) is strictly less than $2c+O(c^2)$, and so is bounded 
by $2c$ for small $c$ (note that $O(c^2)$ errors are negligible).}

We obtain upper bounds on the contributions of the relevant terms 
in the expansion of $\T$ by replacing three
$\Delta g_f(x,y)$'s and $\Delta g_b(x,y)$'s with $\Delta g_f'(x,y)$ 
and $\Delta g_b'(x,y)$, respectively, and replacing all other
terms with $1$. 

Since all graphons are symmetric, hence Hermitian, their operator
norms are bounded by their $L^2$ norms, so for any 3-chain
\be
 \langle 1 |\Delta g_1' \Delta g_2' \Delta g_3' |1 \rangle \le \|\Delta g_1' \| \| \Delta g_2' \| \|\Delta g_3'\|.\ee
Since $\| \Delta g_b' \|$ and $\| \Delta g_f'\|$ are both $o(1)$ (more
precisely, $O(\sqrt{\T-\E^\ell}))$, the contribution of any 3-chain is
bounded by an $o(1)$ constant times $\| \Delta g_f \|^2$.

As for triangles, $Tr(\Delta g_f'^3) \le \| \Delta g_f' \|^{3} = \|
\Delta g_f \|^3$. Finally, we must estimate the trace of $ \Delta g_f'
\Delta g_f' \Delta g_b'$. But this trace is
\be \langle 1 - \omega | \Delta g_f'  \Delta g_f' |1\rangle + \langle \omega | \Delta g_f' \Delta g_f' |1-\omega \rangle
+ 2c \langle \omega |  \Delta g_f'  \Delta g_f' | \omega \rangle.\ee
Since $\| 1 - \omega\| =\sqrt{c}$, the total is bounded by $(2\sqrt{c} + 2c^2) \| \Delta g_f\|^2$. 

The upshot is that the ratio of $s(g_b + \Delta g_f) - s(g_b)$ and
$\T(g_b+\Delta g_f)-\T(g_b)$ is the same as that computed for $k$-stars
(up to an overall factor of $n_k \E^{\ell-k}$), 
plus an $o(1)$ correction. But that ratio
was bounded by a constant $\beta_0 < \beta$. Restricting attention to values of
$\T$ for which the correction is smaller than $(\beta-\beta_0)/2$, we
still obtain the result that having a non-zero $\Delta g_f$ is a less
efficient way of generating additional $\T$ than simply changing
$c$. Thus the optimizing graphon is exactly bipodal.

Once bipodality is established, 
uniquenesss follows exactly as in the proof of Theorem \ref{thm1}. 
The difference between $\Delta \T$ and $n_k \E^{\ell-k} \Delta \T_k$ is 
of order $c^{3/2}$, and so does not affect the linearization of the optimality
equations at $c=0$. 

\section{Linear combinations of $k$-stars}
\label{SEC:linear-com}

We proved Theorem \ref{Main-Thm-Simple} by first showing that $k$-star 
models have the desired behavior, and then showing that, for an arbitrary
$k$-starlike graph $H$, $\Delta \T$ is well-approximated by a multiple
of $\Delta \T_k$, so the model with densities of edges and $H$ behaves 
essentially the same as a model with densities of edges and $k$-stars. 

To prove Theorem \ref{Main-Thm-Complex}, we consider in this section
a family of models
in which we can prove bipodality and uniqueness of entropy maximizers 
directly, as we did for $k$-stars. In the next section, we will show
how to approximate a model with an arbitrary $H$ with a model in this family.  

Let $h(x) = \sum_{k\ge 1} a_k x^k$ be a polynomial with non-negative coefficients and degree $\ge 2$. 
Let $\T = \sum a_k \T_k$, and consider graphs with fixed edge density
$\E$ and fixed $\T$. In \cite{KRRS} it was proved that  the entropy-maximizing
graphons in such models are always multipodal. 

Most of the analysis of $k$-star models carries over to positive linear
combinations, and so will only be sketched briefly. We will provide complete
details where the arguments differ.

In analogy to our earlier development, let $\psi(\E, \tilde \E) = N/D$, 
where 
\begin{eqnarray} N(\E, \tilde \E) &=& 
2[S_0(\tilde \E)-S_0(\E) - (\tilde \E-\E) S_0'(\E)], \cr 
D(\E, \tilde \E) &=& 
h(\tilde \E)-h(\E) - (\tilde \E-\E) h'(\E).
\end{eqnarray}
Since $h''(x)$ is positive for $x>0$, $D$ is only zero when $\tilde \E=\E$, 
and we fill in that removable singularity in $\psi$ by defining
$\psi(\E,\E) = 2 S_0''(\E)/h''(\E)$.

\begin{theorem} \label{Thm-Linear-Combination}
For all but finitely many values of $\E$, there is a $\T_0 > h(\E)$ such that,
for $\T \in (h(\E), \T_0)$, the entropy-optimizing graphon is bipodal and 
unique, with data varying analytically with $\E$ and $\T$. As $\T$ approaches
$h(\E)$ from above, $p_{22} \to \E$, $p_{12}$ approaches a point $\tilde \E$
where $\psi'(\E,\tilde \E)=0$, $p_{11}$ satisfies $S_0'(p_{11})=2S_0'(p_{12})
- S_0'(p_{22})$ and $c \to 0$ as $O(\Delta \T)$. 
\end{theorem}

\begin{proof}

For a multipodal graphon, $\T(g) = \sum c_i h(d_i)$. After eliminating
$\gamma$, the optimality equations become
\begin{align}
S_0'(p_{ij}) &= \alpha + \beta (h'(d_i) + h'(d_j))/2,\\
2\sum_{j=1} c_j (S_0(p_{ij})\!-\!S_0(p_{Mj})) &= 2\alpha (d_i\!-\!d_M) + \beta 
\big [h(d_i)\!-\!h(d_M) +
\sum_{j=1}^M c_j 
h'(d_j) (p_{ij} \!-\! p_{Mj}) \big]. 
\end{align}
As before, we distinguish between Type I clusters that are small and 
Type II clusters that have $d_i \approx \E$. Summing the optimality equations over
$j$ of Type II, and approximating $d_j$ by $\E$, we obtain the equations
\begin{eqnarray}
S_0'(d_i) & = & \alpha + \beta \left ( h'(d_i) + h'(\E) \right)/2, \\
S_0(d_i)-S_0(\E) & = & \alpha(d_i-\E) + \beta \left [
h(d_i) - h(\E) + h'(\E)(d_i-\E) \right ],
\end{eqnarray}
that are accurate to within $o(1)$. We use the first equation, with $i=M$ {(a type II cluster)}, to
solve for $\alpha$, and plug it into the equations for $i<M$ to get
\begin{eqnarray}
  2 (S_0'(d_i)-S_0'(\E)) & = & \beta (h'(d_i) - h'(\E)), \\
  2 [S_0(d_i) -S_0(\E) - S_0'(\E)(d_i-\E)] & = & 
\beta [h(d_i)-h(\E) - h'(\E)(d_i - \E)],
\end{eqnarray}
again to within $o(1)$. {As before in the proof of Theorem \ref{thm1}, this} 
implies that either $d_i \approx \E$ or 
that $\psi(\E, d_i)$ is maximized with respect to $d_i$. 

Unlike in the $k$-star case, it is not true that $\psi'(\E,\tilde \E)$
has a unique solution for each $\E$. However, it remains true that
$\psi(\E,\tilde \E)$ has a unique global maximizer (w.r.t. $\tilde
\E$) for all but finitely many values of $\E$. Since the equations
defining multiple maxima are analytic, they must be satisfied either
for all $\E$ or for only finitely many $\E$. But it is straightforward to 
check that there is only one maximizer when $\E$ is sufficiently small, since
then $h(\E)$ and $h'(\E)$ are dominated by the lowest order term in the 
polynomial. 

Thus, for all but finitely many values of $\E$, the values of $d_i$
must all either approximate $\E$ or the unique value of $\tilde \E$
that maximizes $\psi(\E, \tilde \E)$. This allows for a re-segregation
of the clusters into Type I (with $d_i$ close to $\tilde \E$) and Type
II (with $d_i$ close to $\E$) and yields a graphon that
is approximately bipodal.   Step 2 of the proof of Theorem
\ref{thm1}, proving that the optimizing graphon is exactly
bipodal with data of the desired form, then procedes exactly as before.

What remains is showing that the optimizing graphon is unique by linearizing
the exact optimality equations for bipodal graphons near $c=0$. 
These equations are:
\begin{eqnarray} 
  S_0'(p_{11}) & = &  \alpha + \beta h'(d_1), \cr
  S_0'(p_{12}) & = & \alpha + \beta(h'(d_1) + h'(d_2))/2, \cr 
  S_0'(p_{22}) & = & \alpha + \beta h'(d_2), \cr
  \frac{\partial S}{\partial c} & = & \alpha \frac{\partial \E}{\partial c} + \beta \frac{\partial \E}{\partial c}, \cr 
  \E & = & \E_0, \cr 
  \T & = & \T_0.
\end{eqnarray}

Using the second and third equations to eliminate $\alpha$ and $\beta$ gives:
\vskip1truein
\begin{eqnarray} \label{alphabeta-general}
\alpha & = & \frac{2 h'(d_2) S_0'(p_{12})-S_0'(p_{22}) (h'(d_2)+h'(d_1))}
{h'(d_2) - h'(d_1)}, \cr 
&& \cr 
\beta & = & \frac{2(S_0'(p_{22}) - S_0'(p_{12}))}{h'(d_2)-h'(d_1)}.
\end{eqnarray}
We also have $\alpha = S_0'(p_{22})-\beta h'(d_2)$ and 
$S_0'(p_{11}) = 2S_0'(p_{12})-S_0'(p_{22})$. Note that
\be \frac{\partial \alpha}{\partial p_{12}} = 
-\beta c h''(d_2)-h'(d_2) 
\frac{\partial \beta}{\partial p_{12}} \Rightarrow -h'(p_{22})  
\frac{\partial \beta}{\partial p_{12}}\ee
as $c \searrow 0$. 

We define $\vec f$ as before, with $f_3=\E$ and $f_4=\T$, and compute
\begin{eqnarray}
df_3 &=& (c^2, 2c(1-c), (1-c)^2, 2cp_{11}+2(1-2c)p_{12}-2(1-c)p_{22}) \cr 
& \Rightarrow&  (0,0,1,2(p_{12}-p_{22})) ,\cr 
df_4 & = & (c^2 h'(d_1), c(1-c)(h'(d_1)+h'(d_2)), (1-c)^2 h'(d_2), \cr 
&& \qquad \qquad h(d_1)-h(d_2) + ch'(d_1)(p_{11}-p_{12}) + h'(d_2)(p_{12}-p_{22})) \cr 
& \Rightarrow & (0,0, h'(p_{22}), h(p_{12})-h(p_{22})+h'(p_{22})(p_{12}-p_{22})).
\end{eqnarray}
The lower right block of $d\vec f$ then gives a contribution of 
$h(p_{12})-h(p_{22}) + h'(p_{22})(p_{12}-p_{22}) - 2h'(p_{22})(p_{12}-p_{22}) 
= h(p_{12})-h(p_{22}) - h'(p_{22})(p_{12}-p_{22})=D(p_{22},p_{12})$. 

As before, $\partial f_2/\partial p_{11} = 0$ when $c=0$, so 
$\det(d\vec f) = S_0''(p_{11})(h(p_{11})-h(p_{22}) - h'(p_{22})(p_{12}-p_{22}))
\partial f_2/\partial p_{11}.$ Now
\be
\frac{\partial f_2}{\partial p_{12}} = 
\frac{\partial^2 S}{\partial c \partial p_{12}} - \alpha   
\frac{\partial^2 \E}{\partial c \partial p_{12}} - \beta   
\frac{\partial^2 \T}{\partial c \partial p_{12}} 
- \frac{\partial \alpha}{\partial p_{12}} \frac{\partial \E}{\partial c}
- \frac{\partial \beta}{\partial p_{12}} \frac{\partial \T}{\partial c}.\ee
Since $\alpha$ and $\beta$ are independent of $c$, the first three terms
are 
\be
 \frac{\partial}{\partial c} \left (
\frac{\partial S}{\partial p_{12}} - \alpha   
\frac{\partial \E}{\partial p_{12}} - \beta   
\frac{\partial \T}{\partial p_{12}} \right ) 
= \frac{\partial}{\partial c} (0)=0,
\ee
by the second variational equation. This leaves 
\be
 \partial f_2/\partial p_{12} = (h'(p_{22}) (2p_{12}-2p_{22}) - (h(p_{12})-h(p_{22}) + h'(p_{22})(p_{12}-p_{22}))) \partial \beta/
\partial p_{12}.\ee
Combining with our earlier results, we have:
\be
 \det(d \vec f) = -S_0''(p_{11}) D(p_{22},p_{12})^2 
\frac{\partial \beta}{\partial p_{12}}.\ee    
The expression $D(p_{22},p_{12}) = h(p_{12})-h(p_{22}) - h'(p_{22})
(p_{12}-p_{22})$ has a double root at $p_{12}=p_{22}$ and is nonzero elsewhere,
thanks to the monotonicity of $h'$.

As a last step, we consider when $\partial \beta/\partial p_{12}$ can be 
zero. Since $\beta = N'/D'$, 
we are interested in when $(N'/D')'=0$. But that is equivalent to 
having $N''/D'' = N'/D'$. 
Since we already have $N/D=N'/D'$, this means that 
$\psi''=(N/D)''=0$. 
Since we are looking at the value of $\tilde \E$ that maximizes $\psi$, 
having $\psi'=\psi''=0$ would imply $\psi'''=0$ (or else $\tilde \E$ would
only be a point of inflection, and not a local maximum).
But if $(N/D)'=(N/D)''=(N/D)'''=0$, then $N/D=N'/D' = N''/D'' = N'''/D'''$. 
Note that $N''$, $N'''$, $D''$ and $D'''$ are functions of $\tilde e$ only, and
are rational functions:
\begin{eqnarray}
N'' & = & 2 S_0''(\tilde e) = \frac{-1}{\tilde e} - \frac{1}{1-\tilde e}, \cr 
N''' & = & 2 S_0'''(\tilde e) = \frac{1}{\tilde e^2} 
- \frac{1}{(1-\tilde e)^2}, \cr 
D'' & = & h''(\tilde e), \cr 
D''' & = & h'''(\tilde e).
\end{eqnarray}
Setting $D''N'''=D'''N''$ gives a polynomial equation for $\tilde \E$, 
which has only finitely many roots. Since the equation $\psi'=0$ is 
symmetric is $\E$ and $\tilde \E$, $\tilde \E$ determines
$\E$, so there are only finitely many values of $\E$ for which
$\partial \beta/\partial p_{12}$ is
zero. 

In summary, we exclude the finitely many values of $\E$ for which 
$\psi$ achieves its maximum more than once, and the finitely many values
of $\E$ for which $\partial \beta/\partial p_{12}=0$. For all other values of 
$\E$, the optimizing graphon is bipodal of the prescribed form and unique.  

\end{proof}

\section{Proof of Theorem \protect{\ref{Main-Thm-Complex}}}
\label{SEC:complex}

The proof has three steps.

\begin{itemize}[leftmargin=\dimexpr 26pt+6mm]
\item[Step 1.] Showing that, for fixed $\E$, $\Delta \T$ can be approximated 
by the change in a positive linear combination of $\T_k$'s, as studied in
the last section. 
\item[Step 2.] Defining a set $B_H \subset (0,1)$ of ``bad values'', determined
by analytic equations, such that for all $\E \not \in B_H$ and for $\T$ close
enough to $\E^\ell$, the optimizing graphon is unique and bipodal and of the
desired form. 
\item[Step 3.] Showing that $B_H$ is finite. 
\end{itemize}

\paragraph{Step 1.} 
This is a repetition of the proof of Lemma \ref{LEM:close}. In the
expansion of $\Delta \T$, we get a contribution $n_k \E^{\ell-k} \Delta \T_k$
from diagrams where all the edges associated with $\Delta g$ 
are connected to a vertex of degree $k$, where $n_k$ is the number
of vertices of $H$ of degree $k$. Summing over $k$, and bounding
the remaining terms by $O(\| \Delta g\|^3)$, as before, we have 
\be \Delta \T = \sum_k n_k \E^{\ell-k} \Delta \T_k + O(\Delta \T^{3/2}).\ee

\paragraph{Step 2.} 
For fixed $\E$, we consider a model whose density is $\sum_k n_k
\E^{\ell-k} \T_k$. As long as $\psi(\E,\tilde \E)$ for this model
achieves its maximum at a unique value of $\tilde \E$, and as long as
$\partial \beta/\partial p_{12} \ne 0$ when $p_{12}$ equals this value of $\tilde\E$, the proofs of Theorems
\ref{Main-Thm-Simple} and \ref{Thm-Linear-Combination} carry over
almost verbatim. 

That is, the model problem has a unique bipodal maximizer by
the reasoning of Theorem \ref{Thm-Linear-Combination}. The entropy
maximizer for the actual problem involving $H$ must approximate the
entropy maximizer for the model problem, and in particular must be
approximately bipodal, and so can be written as $g_b + \Delta g_f$,
where $\Delta g_f$ averages to zero on each quadrant. The same arguments
as in the proof of Theorem \ref{Main-Thm-Simple} show that $\Delta g_f$ 
is pointwise small. By a power series expansion, 
$({s(g_b + \Delta g_f)-s(g_b)})/({\T(g_b + \Delta g_f)-\T(g_b)})<\beta$,
so for small $c$ we can increase the entropy by setting $\Delta g_f$ to zero 
and varying the bipodal data to achieve the correct value of $\T$.

\paragraph{Step 3.} For any fixed $\E$, the model problem has only a
finite number of bad values of $\E$, but this is not enough to prove that
$B_H$ is finite.  Rather
\be B_H = \{ \E | \E \hbox{ is one of the bad points for the model with }
a_k = n_k \E^{\ell-k}
\},\ee
where a value of $\E$ is bad for a model 
if either $\psi$ has multiple maxima or if $\partial \beta/\partial p_{12}=0$. 
Since the bad points for any linear combination of $k$-stars depends
analytically on the coefficients of that linear combination, and since these
coefficients are powers of $\E$, 
the set $B_H$ is cut out by analytic equations in $\E$.

As such, $B_H$ is either the entire interval $(0,1)$, or a finite set,
or a countable set with limit points only at 0 and/or 1.  We will show
that neither $0$ nor $1$ is a limit point of $B_H$, implying that
$B_H$ is finite.

Let $k_{max}$ be the largest degree of any vertex in $H$, and
consider the model problem with $h(x) = \sum_{k=2}^{k_{max}} a_k x^k$,
where $a_k = n_k \E^{\ell - k}$.  We begin with some constraints on
the values of $\tilde \E$ for which $\psi'=0$.

\begin{lemma}\label{LEM-E-tilde}
  Suppose that $\psi'(\E,\tilde \E)=0$.  If $\tilde \E=\E$, or if
  $\partial \beta/\partial p_{12}=0$ when $p_{22}=\E$ and
  $p_{12}=\tilde \E$, then $({1}/{2}) \le \tilde \E \le
  ({k_{max}-1})/{k_{max}}$.
\end{lemma}

\begin{proof}[Proof of lemma]
  In both cases we are looking for solutions to $N'' D'''=N'''
  D''$. Since $N'' = 2 S_0''(\tilde \E)$, $N''' = 2 S_0'''(\tilde
  \E)$, $D'' = h''(\tilde \E)$ and $D'''=h'''(\E)$, this equation does
  not involve $\E$ (except insofar as the coefficients of $h$ depend
  on $\E$). We have
  \begin{eqnarray} \frac{2 S_0'''(\tilde \E)}{2S_0''(\tilde \E)} & =
    & \frac{h'''(\tilde \E)}{h''(\tilde \E)}, \cr 
    \frac{1}{1-\tilde \E}
    - \frac{1}{\tilde \E} & =& \frac{h'''(\tilde \E)}{h''(\tilde \E)},
    \cr 
    \frac{2\tilde \E-1}{1-\tilde \E} & =& \frac{\tilde \E
      h'''(\tilde \E)}{h''(\tilde \E)}, \cr 
      \frac{1}{1-\tilde \E} - 2 &
    = & \frac{\sum k(k-1)(k-2) a_k \tilde \E^{k-2}}{\sum k(k-1) a_k
      \tilde \E^{k-2}}.
\end{eqnarray}
The right hand side of the last line is a weighted average of $k-2$
with weights $k(k-1)a_k \tilde \E^{k-2}$, and so is at least zero and
at most $k_{max}-2$. Thus $(1-\tilde \E)^{-1}$ is between 2 and
$k_{max}$ and $\tilde \E$ is between $1/2$ and
$({k_{max}-1})/{k_{max}}$.
\end{proof}

\begin{lemma} If $\psi'(\E, \tilde \E)=0$, and if $\E$ is sufficiently
  close to 1, then $\tilde \E$ is uniquely defined and approaches 0 as
  $\E \to 1$. Likewise, if $\E$ is sufficiently close to 0, then
  $\tilde \E$ is uniquely defined and approaches 1 as $\E \to 0$.
\end{lemma}

\begin{proof} When $\E < 1/2$, or when $\E >
  ({k_{max}-1})/{k_{max}}$, we cannot have $\tilde \E = \E$, so the
  equation $\psi'=0$ is equivalent to $ND'=DN'$ and $\tilde \E \ne
  \E$.  Writing $DN'-ND'=0$ explicitly, and doing some simple algebra,
  yields the equation
\be
S_0'(\E)[h (\tilde \E) - h(\E) - (\tilde \E-\E) h'(\E)] - S'(\tilde \E)[[h (\tilde \E) - h(\E) - (\tilde \E-\E) h'(\tilde \E)] 
+ (S_0(\tilde \E)-S_0(\E))(h'(\tilde \E)-h'(\E)) = 0.\ee 
If $\E$
approaches 0 or 1 and $\tilde \E$ does not, then the first term
diverges, while the other terms do not, insofar as $S_0'$ has
singularities at 0 and 1 but $S_0$, $h$ and $h'$ do not.  Thus $\tilde
\E$ must go to 0 or 1 as $\E$ goes to 0 or 1.

We next rule out the possibility that both $\E$ and $\tilde \E$ approach 1. 
Suppose that $\E$ is close to 1. We expand both $N$ and $D$ in powers of 
$(\tilde \E - \E)$:
\begin{eqnarray}
N & = & \sum_{m=2}^\infty \frac{2 S_0^{(m)}(\E)}{m!} (\tilde \E - \E)^m \cr 
& = & - \sum_{m=2}^\infty \left ( \frac{1}{(1-\E)^{m-1}} + \frac{(-1)^{m}}
{\E^{m-1}} \right ) \frac{(\tilde \E-\E)}{m(m-1)}, \cr 
D & = & \sum_{m=2}^{k_{max}} \frac{h^{(m)}(\E)}{m!} (\tilde \E - \E)^m,
\end{eqnarray}
where $S_0^{(m)}$ and $h^{(m)}$ denote $m$th derivatives. 
Note that the coefficients of the numerator grow rapidly with $m$, while 
the growth of the coefficients of the denominator depend only on the degree
of $h$. For $\tilde \E > \E > (k_{max} - 1)/{k_{max}}$, $\psi = N/D$ is a 
decreasing function of $\tilde \E$ (that is, negative and 
increasing in magnitude), so we cannot have $\psi'=0$.  Since the equation
$\psi'=0$ is symmetric in $\E$ and $\tilde \E$ (apart from the dependence of 
the coefficients of $h$ on $\E$), we also cannot have 
$\E > \tilde \E > (k_{max}-1)/{k_{max}}$. 

When $\E$ is close to 1, we must thus have $\tilde \E$ close to 0. But then
$N \approx 2 S_0'(\E)$, $D \approx h'(\E)-h(\E)$, $D' \approx - h'(E)$, and 
the equation
\be
2 S_0'(\tilde \E) = N' + 2S_0'(\E) = 2 S_0'(\E) + ND'/D\ee 
determines $S_0'(\tilde \E)$, and therefore $\tilde \E$, uniquely as a function
of $\E$. 

Next we consider $\E \to 0$. If $H$ is 2-{starlike}, then $\psi$ is a multiple
of $\psi_2$, and the result is already known. Otherwise,
it is convenient to define a new polynomial
$\bar h(z) = \sum n_k z^k$, so that $h(x) = \E^\ell \bar h(x/\E)$. 
Then
\begin{eqnarray}
D & = & h(\tilde \E) - h(\E) - h'(\E)(\tilde \E-\E) \cr 
& = & \E^\ell [\bar h(r) - \bar h(1) - \bar h'(1)(r-1)]
\end{eqnarray}
where $r := \tilde \E/\E$. Likewise, 
\be
N  =  - \left [\tilde \E \ln(\tilde \E) - \E \ln(\E) + (1-\tilde \E)
\ln(1-\tilde \E) - (1-\E)(1-\tilde \E)-(\tilde \E-\E)(\ln(\E)-\ln(1-\E))
\right ] 
\ee
Since $\E$ and $\tilde \E$ are small, we can approximate 
$\ln(1-\E)$ and $\ln(1-\tilde \E)$ as $-\E$ and $-\tilde \E$,
respectively, giving
\be N \approx - \E[r \ln r - r + 1] + \E^2(r-r^2) \ee
Note that the ratio $\psi = N/D$ is negative. Since $\bar h$ is a polynomial
of degree at least 3, $D$ grows faster than $N$ as $r \to \infty$, so we can
always increase $\psi$ by taking larger and larger values of 
$r = \tilde \E/\E$. This argument only breaks down when the approximation
$\ln(1-\tilde E) \approx -\tilde \E$ breaks down, i.e. at values of 
$\tilde \E$ that are no longer close to 0. Thus we cannot have $\tilde \E$ and 
$\E$ both close to zero. 

Finally, if $\E$ is close to 0 and $\tilde \E$ is close to 1, then 
$h(\E)$ and $h'(\E)$ are close to zero, while 
$h(\tilde \E)$ is close to a multiple of $x^{k_{max}}$, since the coefficient of 
$x^{k_{max}}$ is $O(1/\E)$ larger than any other coefficient. Thus 
$\psi$ behaves like $\psi_{k_{max}}$, and has a unique maximizer. 
\end{proof}

We have shown that when $\E$ is close to 0 or 1, 
$\psi$ has a unique maximizer.  Furthermore,
$\tilde \E$ is not between $1/2$ and $({k_{max}-1})/{k_{max}}$,
so $\partial \beta/\partial p_{12} \ne 0$.  So $\E \not \in B_H$,
completing Step 3 and the proof of Theorem \ref{Main-Thm-Complex}.

\section{Conclusions}

We have shown that just above the ER curve, entropy maximizing
graphons, constrained by the densities of edges and any one other
subgraph $H$, exhibit the same qualitative behavior for all $H$ and
for (almost) all values of $\E$.  The optimizing graphon is unique and
bipodal.

These results were proven by perturbation theory, using the fact that
the optimizing graphon has to be $L^2$-close to a constant (Erd\H
os-R\'enyi) graphon. Surprisingly, the optimizing graphon is {\em not}
pointwise close to constant. Rather, it is bipodal, with a small
cluster of size $O(\Delta \T)$. As $\Delta \T$ approaches 0, the size
of the small cluster shrinks, but the values of the graphon on each
quadrant do not approach one another. Rather, $p_{22}$ approaches
$\E$, $p_{12}$ approaches the value of $\tilde \E$ that maximizes a
specific function $\psi(\E, \tilde \E)$, and $p_{11}$ satisfies
$S_0'(p_{11}) - 2 S_0'(p_{12}) + S_0'(p_{22})=0$.

Finally, the asymptotic behavior of these graphons as $\T \to \E^\ell$
depends only on the degree sequence of $H$. In particular, the cases
where $H$ is a triangle and when $H$ is a 2-star are asymptotically
the same. This is illustrated in Figure \ref{FIG:Bipodal 2Star Triag}.
Since $\Delta \T_{\text{triangle}} \approx 3 \E \Delta \T_2$, the
optimizing graphon for the 2-star model with $\E = 0.4$ and $\Delta
\T_2=0.002$ should resemble the optimizing graphon for the triangle
model with $\E=0.4$ and $\Delta \T_{\text{triangle}}=0.0024$. These
optimizing graphons are obtained using the algorithms we developed
in~\cite{RRS} \emph{without assuming bipodality}. Numerical estimates
indicate that the optimizing graphons are not exactly the same, thanks
to $O(\Delta \T_2^{3/2})$ corrections to $\Delta
\T_{\text{triangle}}$, but are still qualitatively similar.
 
\begin{figure}[ht]
\centering
\includegraphics[angle=0,width=0.4\textwidth]{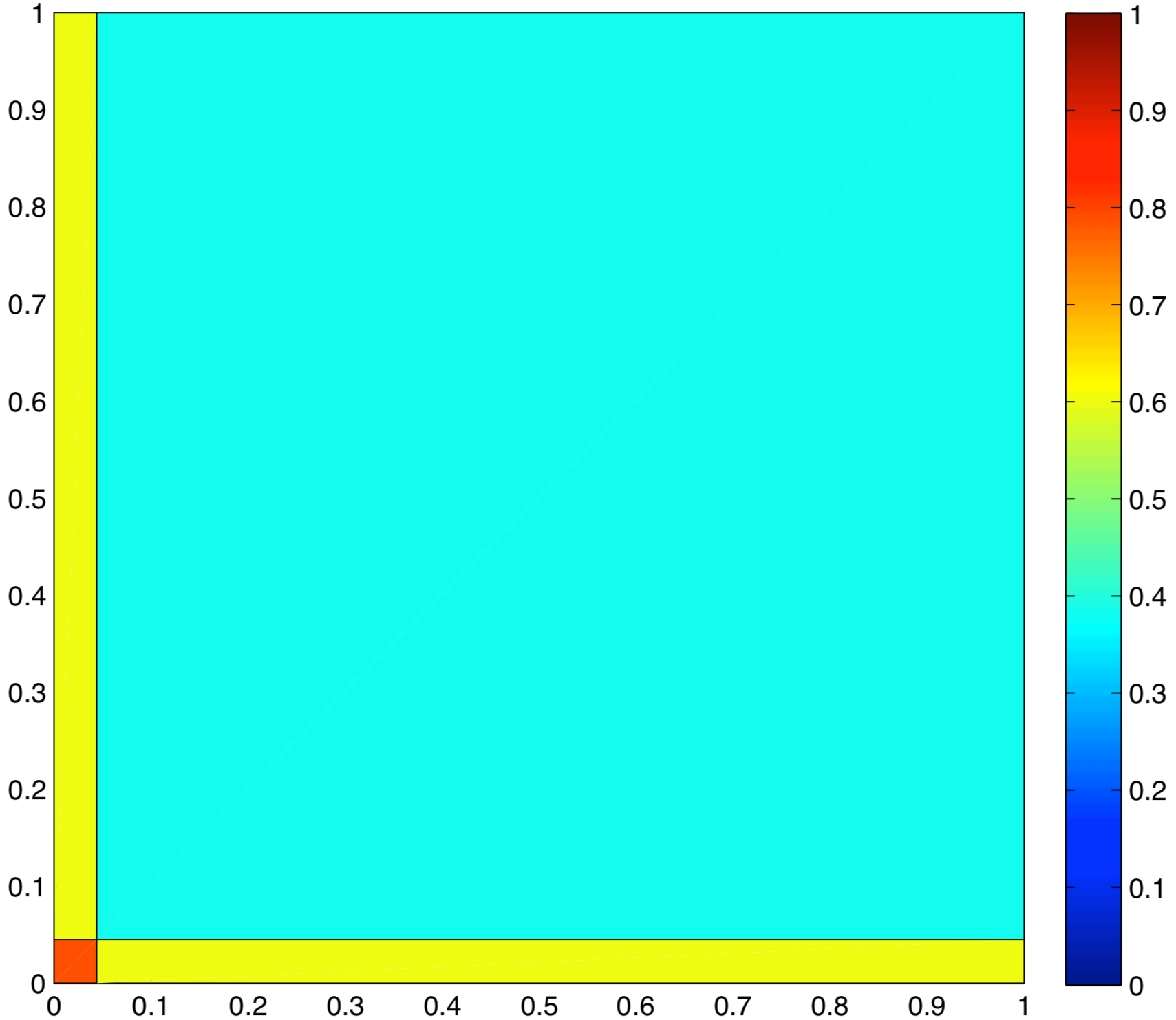}\hskip 1cm 
\includegraphics[angle=0,width=0.4\textwidth, height=0.345\textwidth]{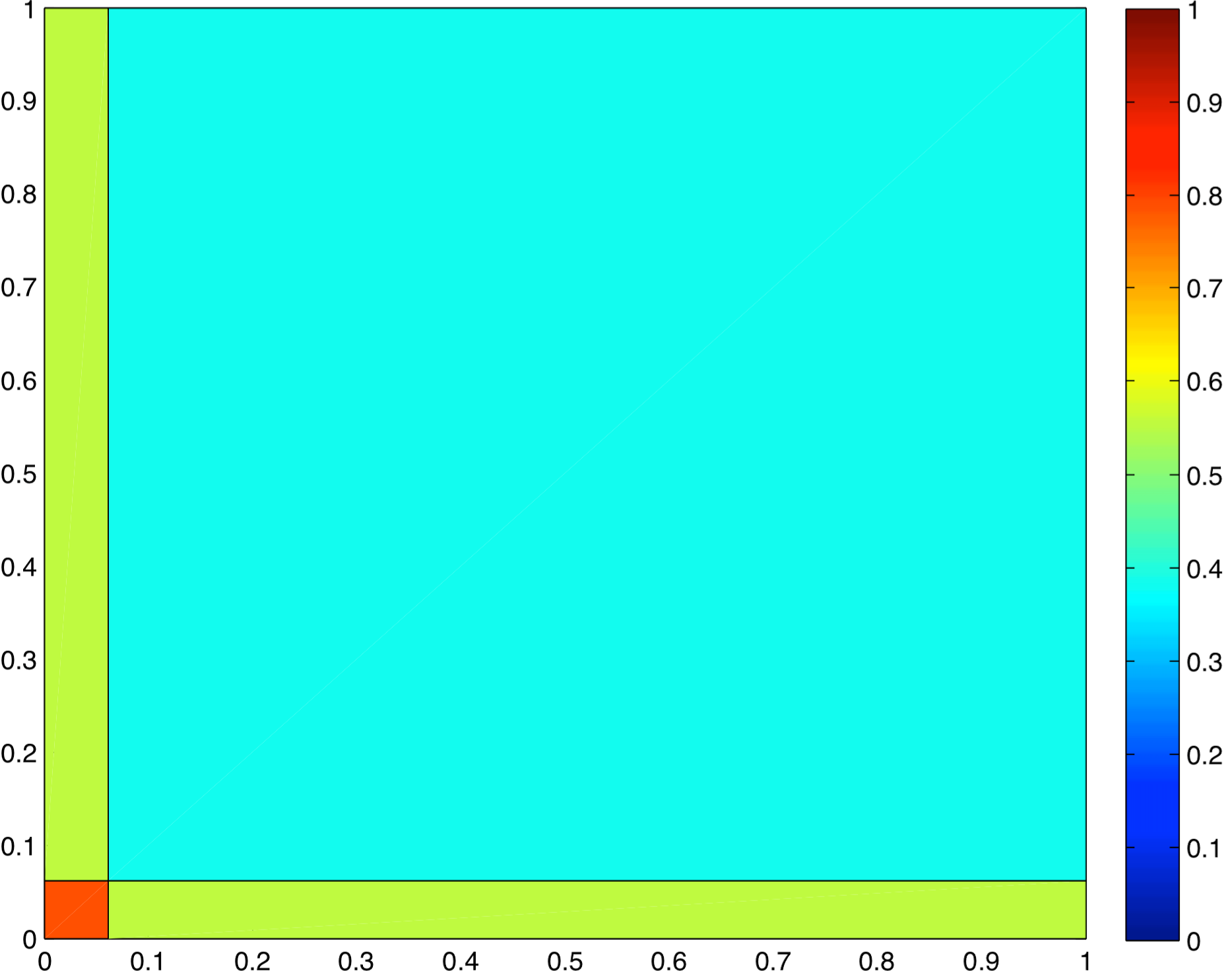} 
\caption{Numerical estimates of the optimizing graphon for the 2-star
  model with $\E=0.4$ and $\T_2=0.1620$ (left) and the optimizing
  graphon for the triangle model with $\E=0.4$ and
  $\T_{\text{triangle}}=0.0664$ (right).  (Although theoretically we
  have not tried to prove that $\Delta\tau_2 = 0.002$ is small enough
  to fit into the interval provided by
  Theorem~\protect{\ref{Main-Thm-Simple}}, numerically it appears to
  be the case.)}
\label{FIG:Bipodal 2Star Triag}
\end{figure}

\section*{Acknowledgments}

The computational results shown in this work were obtained on the computational facilities in the Texas Super Computing Center
(TACC). We gratefully acknowledge this computational support. R. Kenyon was partially supported by the Simons Foundation grant 327929.
This work was also partially supported by NSF grants DMS-1208191, DMS-1509088, DMS-1321018 and DMS-1101326.



\end{document}